\documentclass[reqno]{amsart}

\usepackage{amsmath,mathtools,amssymb,amsthm}
\mathtoolsset{showonlyrefs}
\allowdisplaybreaks
\usepackage{xcolor}
\usepackage{graphicx}
\usepackage[inline]{enumitem}
\usepackage{algorithm}
\usepackage{algpseudocode}
\usepackage{geometry}
\numberwithin{equation}{section}

\usepackage[
  backend=biber,
  style=authoryear-comp,
  natbib=true,
  sorting=nyt,
  doi=false,
  isbn=false,
  url=true,
  eprint=true,
  giveninits=true,
  useprefix=true,
  uniquename=false,
  uniquelist=false,
  maxbibnames=99,
  maxcitenames=2,
  dashed=false,
  backref=false
]{biblatex}
\AtEveryBibitem{\clearlist{language}\clearfield{note}\clearfield{issn}\clearfield{isbn}}
\DeclareNameAlias{sortname}{family-given}
\DeclareNameAlias{author}{family-given}
\DeclareNameAlias{editor}{family-given}
\DeclareDelimFormat{finalnamedelim}{\addspace\&\space}
\DeclareDelimFormat[bib]{nametitledelim}{\addspace}
\renewbibmacro{in:}{\ifentrytype{article}{}{\printtext{\bibstring{in}\intitlepunct}}}
\DeclareFieldFormat[article]{volume}{\mkbibbold{#1}}
\DeclareFieldFormat{pages}{#1}
\DeclareFieldFormat[incollection,inproceedings]{pages}{pp.~#1}
\DeclareFieldFormat{eid}{#1}
\renewcommand*{\bibpagespunct}{\addcomma\space}
\renewbibmacro*{journal+issuetitle}{%
  \usebibmacro{journal}%
  \ifhmode\spacefactor1000 \fi
  \setunit{\addcomma\space}%
  \printfield{volume}%
  \setunit{\addcomma\space}%
  \printfield{pubstate}%
  \clearfield{pubstate}%
  \newunit}
\renewbibmacro*{publisher+location+date}{%
  \printlist{location}%
  \iflistundef{location}{}{\spacefactor1000 }%
  \iflistundef{publisher}{\setunit*{\addcomma\space}}{\setunit*{\addcolon\space}}%
  \printlist{publisher}%
  \newunit}
\renewbibmacro*{note+pages}{%
  \setunit{\bibpagespunct}%
  \iffieldundef{pages}{\printfield{eid}}{\printfield{pages}}%
  \newunit}
\DefineBibliographyStrings{english}{inpress = {to appear}}
\DeclareFieldFormat[article]{url}{}
\DeclareBibliographyDriver{online}{%
  \usebibmacro{bibindex}%
  \usebibmacro{begentry}%
  \usebibmacro{author/editor+others/translator+others}%
  \setunit{\printdelim{nametitledelim}}\newblock
  \usebibmacro{title}%
  \newunit\newblock
  \usebibmacro{eprint}%
  \newunit\newblock
  \usebibmacro{finentry}}
\AtBeginBibliography{\emergencystretch=2em}
\DeclareFieldFormat{eprint:arxiv}{%
  arXiv\addcolon
  \ifhyperref
    {\href{https://arxiv.org/abs/#1}{\nolinkurl{#1}}}
    {\nolinkurl{#1}}%
  \iffieldundef{eprintclass}{}{\addspace\texttt{\mkbibbrackets{\thefield{eprintclass}}}}}
\newbibmacro*{titlelink}[1]{%
  \iffieldundef{doi}
    {\iffieldundef{url}
       {\iffieldundef{eprint}{#1}{\href{https://arxiv.org/abs/\thefield{eprint}}{#1}}}
       {\href{\thefield{url}}{#1}}}
    {\href{https://doi.org/\thefield{doi}}{#1}}}
\DeclareFieldFormat{title}{\usebibmacro{titlelink}{\mkbibemph{#1}}}
\DeclareFieldFormat[article,inproceedings,incollection,online,software,unpublished,misc]{title}{\usebibmacro{titlelink}{#1}}

\usepackage[pdfencoding=auto, psdextra]{hyperref}
\hypersetup{
    plainpages=false,
    unicode=true,
    pdftoolbar=true,
    pdfmenubar=true,
    pdffitwindow=true,
    pdftitle={Oriented point-cloud varifolds for estimating the perimeter and computing its Wasserstein gradient flow},
    pdfauthor={Tokuhiro Eto and Noboru Isobe},
    pdfsubject={Mullins-Sekerka flow, oriented varifolds, point clouds},
    pdfcreator={Creator},
    pdfproducer={Producer},
    pdfkeywords={oriented varifold; perimeter; point cloud; Mullins-Sekerka flow; minimizing movements},
    pdfnewwindow=true,
    colorlinks=true,
    linkcolor=blue,
    citecolor=blue,
    filecolor=blue,
    urlcolor=magenta,
    urlbordercolor=0 1 1
}

\usepackage{nicefrac}
\makeatletter
\renewcommand\paragraph{\@startsection{paragraph}{4}{\z@}{\medskipamount}{-\fontdimen2\font}{\normalfont\itshape}}
\makeatother
\usepackage{zref-clever}
\zcsetup{cap=true,abbrev=false,nameinlink=true,ref=default,endrange=ref}
\newcommand{\crefname}[3]{%
  \zcRefTypeSetup{#1}{Name-sg={#2},name-sg={#2},Name-pl={#3},name-pl={#3}}%
}
\let\cref\zcref
\let\Cref\zcref

\newtheorem{theorem}{Theorem}[section]
\newtheorem{lemma}[theorem]{Lemma}
\newtheorem{corollary}[theorem]{Corollary}
\newtheorem{proposition}[theorem]{Proposition}

\theoremstyle{definition}
\newtheorem{definition}[theorem]{Definition}
\newtheorem{assumption}[theorem]{Assumption}
\theoremstyle{remark}
\newtheorem{remark}[theorem]{Remark}
\AddToHook{env/lemma/begin}{\zcsetup{reftype=lemma}}
\AddToHook{env/corollary/begin}{\zcsetup{reftype=corollary}}
\AddToHook{env/proposition/begin}{\zcsetup{reftype=proposition}}
\AddToHook{env/conjecture/begin}{\zcsetup{reftype=conjecture}}
\AddToHook{env/definition/begin}{\zcsetup{reftype=definition}}
\AddToHook{env/assumption/begin}{\zcsetup{reftype=assumption}}
\AddToHook{env/remark/begin}{\zcsetup{reftype=remark}}
\AddToHook{env/equation/begin}{\zcsetup{reftype=}}
\AddToHook{env/enumerate/begin}{\zcsetup{reftype=}}
\makeatletter
\AtBeginDocument{%
  
  \newenvironment{proof}[1][\proofname]{\par
    \pushQED{\qed}%
    \normalfont\topsep6\p@\@plus6\p@\relax
    \labelsep5\p@
    \trivlist
    \item[\hskip\labelsep\bfseries\itshape #1\@addpunct{.}]\ignorespaces
  }{\popQED\endtrivlist\@endpefalse}%
}
\makeatother

\newcommand{\Haus}[1]{\mathcal{H}^{#1}}
\newcommand{\dH}[1]{\d\Haus{#1}}
\newcommand{\R}{\mathbb{R}}

\newcommand{\mlfr}[1]{\psi_{#1}}
\DeclareMathOperator{\supp}{supp}

\DeclareMathOperator{\dist}{dist}
\DeclareMathOperator{\Div}{div}

\newcommand{\cutoff}[1]{\eta_{#1}}

\newcommand{\PP}{\mathbb P}
\newcommand{\EE}{\mathbb E}

\newcommand{\Lip}{\mathrm{Lip}}

\newcommand{\measure}[0]{|\mu|}
\newcommand{\measureVector}[0]{\vv{\mu}}
\newcommand{\measureMassEmpirical}[1]{|\mu|_{#1}}
\newcommand{\measureVectorEmpirical}[1]{\measureVector_{#1}}

\newcommand{\characteristics}[1]{\chi_{#1}}

\usepackage{physics2}
\usephysicsmodule{ab}
\usephysicsmodule{ab.braket}
\usephysicsmodule{diagmat}
\usephysicsmodule[showleft=2,showtop=2]{xmat}
\usepackage[normal]{fixdif}
\usepackage[italic]{derivative}
\usepackage[d]{esvect}
\newcommand{\abb}[1]{\bab{#1}}
\NewDocumentCommand{\abs}{s o m}{%
  \IfBooleanTF{#1}{\vab*{#3}}{%
    \IfValueTF{#2}{\vab[#2]{#3}}{%
      \mathchoice{\vab{#3}}{\vab*{#3}}{\vab*{#3}}{\vab*{#3}}}}}
\NewDocumentCommand{\norm}{s o m}{%
  \IfBooleanTF{#1}{\Vab*{#3}}{%
    \IfValueTF{#2}{\Vab[#2]{#3}}{%
      \mathchoice{\Vab{#3}}{\Vab*{#3}}{\Vab*{#3}}{\Vab*{#3}}}}}

\newcommand{\Z}{\mathbb{Z}}

\newcommand{\order}[1]{\mathrm{O}(#1)}
\newcommand{\Expectation}[1]{\EE \abb{#1}}
\newcommand{\ExpectationGiven}[2]{\EE\left[#1\,\middle|\,#2\right]}
\newcommand{\ExpectationSample}[1]{\ExpectationGiven{#1}{X_1,\dots,X_N}}

\def\Set#1{\Setdef#1\Setdef}
\def\Setdef#1|#2\Setdef{\left\{#1\,\;\mathstrut\vrule\,\;#2\right\}}%
\makeatletter
\newcommand{\fmd}[0]{\overline n}
\newcommand{\sci}[2]{#1\times10^{#2}}
\newcommand{\fmdbranch}[1]{\overline n_{#1}}

\DeclareMathOperator{\Per}{Per}
\DeclareMathOperator{\median}{median}
\DeclareMathOperator{\rank}{rank}
\newcommand{\transpose}{^{\top}}

\DeclareMathOperator*{\argmin}{{arg~min}}

\DeclareMathOperator{\BL}{BL}
\zcRefTypeSetup{equation}{Name-sg={},name-sg={},Name-pl={},name-pl={},namesep={}}
\crefname{theorem}{Theorem}{Theorems}
\crefname{definition}{Definition}{Definitions}
\crefname{problem}{Problem}{Problems}
\crefname{fact}{Fact}{Facts}
\crefname{proposition}{Proposition}{Propositions}
\crefname{lemma}{Lemma}{Lemmas}
\crefname{conjecture}{Conjecture}{Conjectures}
\crefname{assumption}{Assumption}{Assumptions}
\crefname{condition}{Condition}{Conditions}
\newlist{steps}{enumerate}{1}
\setlist[steps]{label=\rm(\roman*),ref=(\roman*)}
\crefname{step}{Step}{Steps}
\zcsetup{countertype={enumi=condition,stepsi=step}}
\crefname{claim}{Claim}{Claims}

\crefname{remark}{Remark}{Remarks}
\crefname{example}{Example}{Examples}
\crefname{corollary}{Corollary}{Corollaries}
\crefname{subsubsection}{Subsubsection}{Subsubsections}
\crefname{subsection}{Subsection}{Subsections}
\crefname{section}{Section}{Sections}
\crefname{chapter}{Chapter}{Chapters}
\crefname{table}{Table}{Tables}
\crefname{figure}{Figure}{Figures}
\zcRefTypeSetup{figure}{Name-sg-ab={Fig.},name-sg-ab={Fig.},Name-pl-ab={Figs},name-pl-ab={Figs}}
\crefname{algorithm}{Algorithm}{Algorithms}

\crefname{myenumi}{item}{items}
\crefname{myenumii}{item}{items}
\crefname{myenumiii}{item}{items}
\makeatother

\newcommand{\paperabstract}{%
This paper develops a mesh-free numerical method for the one-phase Mullins--Sekerka flow, in which the interface bounding a set moves as the Wasserstein gradient flow of its perimeter.
One difficulty in computing this flow is that two pieces of the interface with opposite normals form a hidden boundary when components merge.
Our approach represents the interface by an oriented point cloud, which we treat as an oriented varifold.
We introduce a quantity that measures how much nearby normals cancel each other and use it to detect the hidden boundary.
Based on this quantity, we construct a perimeter estimator, prove that its expected error decays at an algebraic rate in the number of points, also in the presence of a hidden boundary, and use it in a minimizing movement scheme for the flow in two dimensions.
Numerical experiments show that the scheme computes mergers of components through the hidden boundary, while the perimeter term in the scheme decreases at every step and the area is nearly conserved.
As hidden boundaries also appear in other perimeter-driven interface problems, we expect our oriented representation to be applicable to them as well.}
\begin{document}

\title[Oriented point-cloud varifolds]{Oriented point-cloud varifolds for estimating the perimeter and computing its Wasserstein gradient flow}
\author{Tokuhiro Eto}
\address{Universit\'e Claude Bernard Lyon 1, CNRS, Centrale Lyon, INSA Lyon, Universit\'e Jean Monnet, ICJ UMR5208, Villeurbanne, France}
\email{eto@math.univ-lyon1.fr}
\author{Noboru Isobe}
\address{RIKEN Center for Advanced Intelligence Project, Tokyo, Japan}
\email{noboru.isobe@riken.jp}
\begin{abstract}
\paperabstract
\end{abstract}
\maketitle
\section{Introduction}
In this paper, we develop a variational numerical scheme for a perimeter-minimizing interface evolution.
Our motivating model is the one-phase Mullins--Sekerka flow \citep{MullinsSekerka63}, the quasi-static interface motion in which the normal velocity is the flux of a potential that is harmonic inside the phase and whose boundary values are the curvature.
It arises as a sharp-interface limit of the Cahn--Hilliard equation with disparate mobilities \citep{KroemerLaux22}.
Otto's variational formulation~\citep{Otto98} interprets the one-phase flow as the gradient flow of the perimeter with respect to the Wasserstein metric.
The corresponding minimizing movement scheme is formally represented as
\begin{equation}
\label{eq:mms-otto}
        E^{n} \in \argmin_{E\subset\Omega}\Bab{\Per(E;\Omega) + \frac{1}{2h}W_2^2(\chi_E\d x, \chi_{E^{n-1}}\d x)},
\end{equation}
where $\Omega$ is a bounded domain in $\R^d$ with $d\geq 2$, $\Per(E;\Omega)$ is the perimeter of a set $E$ in $\Omega$, $h>0$ is a time step, $\chi_E$ is the characteristic function of $E$, and $W_2$ is the Wasserstein distance between two finite measures on $\R^d$ of equal total mass.
The interface is the boundary $\partial E$ of the phase $E$.
\citet{ChambolleLaux21} showed that a subsequence of \((E^n)_n\) converges as $h\to0$ to a limit satisfying a distributional formulation of the flow.

One difficulty here is that the interface can form what \citet{Roger2004,Roger2005} calls the \emph{hidden boundary}.
Intuitively, it is the part of the interface that vanishes when several connected components of $E$ merge in time (see the top of \cref[abbrev]{fig:num-two-ellipses}).
Just before vanishing, this part consists of two pieces that lie on top of each other with opposite normals, and when they vanish, $\Per(E;\Omega)$ can drop discontinuously by twice their $(d-1)$-dimensional measure.
Analytical convergence results often exclude the hidden boundary by assuming that no perimeter is lost in the limit \citep{LS95,LauxOtto16}.
Numerical simulations, however, cannot exclude it because connected components merge in many applications, for example in phase coarsening \citep{AndrewsVoorheesThornton20,Ramani26} and for bubbles in Hele-Shaw cells \citep{KeelerEtAl22}.
They require a numerical method that handles the hidden boundary.

Most existing numerical methods for the Mullins--Sekerka flow deal with the hidden boundary only at the price of a grid on $\Omega$ or a remeshing of the interface.
Implicit representations describe the phase by an auxiliary field on a fixed computational grid, as in \textit{phase-field methods} \citep{FengProhl05,G03}, \textit{level-set methods} \citep{ChenMerrimanOsherSmereka97,HouLiOsherZhao97} together with the implicit boundary integral method \citep{CKT17}, and \textit{threshold-type methods} for the related Muskat problem \citep{JacobsKimMeszaros21}.
For these methods, the hidden boundary causes no difficulty, but they require a grid on $\Omega$ fine enough to resolve the interface.
Perimeter estimates by graph cuts on random samples of $\Omega$ instead of a structured grid are also known \citep{Trillos16,Trillos17}, but they still require samples in $\Omega$ dense enough to resolve the interface.
Explicit representations describe the interface by a moving mesh, that is, by nodes on the interface together with edges between neighbouring nodes, as in front-tracking methods with parametric finite elements \citep{BGN10,R23,EGN24,EGN26-2}, boundary integral methods \citep{HouLowengrubShelley94,BCD95,ZCH96,BB00} and the method of fundamental solutions \citep{SakakibaraYazaki19,E24,E26}.
These methods give direct access to the interface, but a change of topology requires an explicit remeshing of these edges.

The aim of this paper is to handle the hidden boundary with neither a grid on $\Omega$ nor a moving mesh of the interface.
One natural representation of the interface is a point cloud, which consists of nodes without edges.
\citet{ZHH26} develop a point-cloud boundary integral method, which orients the normals from the known interior of a single smooth closed curve.
\citet{BuetRumpf22} represent curves and surfaces by a point-cloud varifold, which assigns a weight and an unoriented tangent space to each point, and compute their mean curvature motion.
Neither method can represent the hidden boundary, in which two pieces of the interface lie on top of each other with opposite normals.

We propose to employ oriented point clouds, a representation of curves and surfaces in computer vision and computer graphics \citep{KazhdanBolithoHoppe06,PengEtAl21,KaltenmarkCharlierCharon17}, for the Mullins--Sekerka flow with the hidden boundary.
An oriented point cloud discretizes the interface $\partial E$ by a set $\{(X_i,U_i)\}_{i=1}^N\subset\Omega\times S^{d-1}$ of pairs of a point $X_i$ and a unit normal $U_i$.
We treat such a cloud as an oriented varifold, namely the weighted sum $V_N\coloneqq\frac1N\sum_{i=1}^N W_i\delta_{(X_i,U_i)}$ of Dirac measures, whose weights $W_i$ are estimated from the cloud itself with a bandwidth $\delta$ by the method of \citet{Boricaud_Buet25}, as detailed in \cref{sec:estimation}.
Oriented varifolds in the spirit of \citet{Hutchinson86} are finite non-negative Radon measures on $\Omega\times S^{d-1}$, and they describe a smooth $\partial E$ by the measure concentrated on the graph $\Set{(x,n_E(x)) | x\in\partial E}$ of its outer unit normal $n_E$.

Our main idea to overcome the difficulty posed by the hidden boundary is to detect it in an oriented varifold $V$ by what we call the \emph{visibility},
\[
\begin{aligned}
q_\sigma[V](x)\coloneqq\frac{\abs{(\psi_\sigma\ast\measureVector)(x)}}{(\psi_\sigma\ast\measure)(x)}, &&
\measure(A)\coloneqq V(A\times S^{d-1}), &&
\measureVector(A)\coloneqq\int_{A\times S^{d-1}}n\d V(x,n),
\end{aligned}
\]
where $\psi_\sigma$ is a mollifier at scale $\sigma>0$, the ratio is set to $0$ where $(\psi_\sigma\ast\measure)(x)=0$, and the weight measure $\measure$ and the current $\measureVector$ of $V$ are defined for measurable $A\subset\Omega$.
The weight measure $\measure$ counts every piece of the interface positively, whereas two pieces with opposite normals cancel in the current $\measureVector$.
Hence $q_\sigma$ is close to $1$ where the normals nearby agree and close to $0$ on the hidden boundary.

As our first contribution, we use the visibility $q_\sigma$ to construct an estimator $\widehat P_{\sigma,\delta,N}$ of $\Per(E;\Omega)$ in \cref{eq:mms-otto}.
More generally, we can use $\widehat P_{\sigma,\delta,N}$ to estimate a generalized perimeter $P(V;\Omega)\coloneqq\abs{\measureVector}(\Omega)$ of an oriented varifold $V$ from which the oriented points are sampled.
This perimeter coincides with $\Per(E;\Omega)$ when $V$ is the oriented varifold of a set $E$ of finite perimeter (\cref{rem:perimeter}).
The estimator is computed from the oriented varifold $V_N$ as
\[
\widehat P_{\sigma,\delta,N}\coloneqq\frac1N\sum_{i=1}^N W_i\,q_\sigma[V_N](X_i).
\]
We expect $\widehat P_{\sigma,\delta,N}$ to approximate $P(V;\Omega)$ because $q_\sigma[V_N]$ discounts the contribution of the hidden boundary, which $P(V;\Omega)$ does not count.
Under geometric assumptions on $V$ and for independent samples, our theoretical results (\cref{thm:quantitative-perimeter-consistency,cor:explicit-rate-with-BB-bandwidth}) show that the expected value of the error of $\widehat P_{\sigma,\delta,N}$
decays at an algebraic rate in $N$ for a choice of $(\sigma,\delta)$, also when $V$ has a hidden boundary.

As our second contribution, we use $q_\sigma$ and $\widehat P_{\sigma,\delta,N}$ to construct a grid- and mesh-free numerical scheme for \cref{eq:mms-otto} with $d=2$ (\cref{sec:application}).
In each step, we replace $\Per(E;\Omega)$ by $\widehat P_{\sigma,\delta,N}$, with the visibility taken from the previous step, and $W_2^2(\chi_E\d x,\chi_{E^{n-1}}\d x)$ by its linearization, which we evaluate by a boundary integral method on the point cloud.
We need additional techniques in the implementation to conserve the areas and barycentres as the flow does, to prevent the points from becoming too non-uniform along the interface and to deal with changes of topology.
We note that the scheme temporarily reconstructs the connectivity of some points from the point cloud to mitigate the ill-conditioning of the minimization at a change of topology.
The numerical results in \cref[abbrev]{fig:num-two-ellipses} show that the scheme can simulate the Mullins--Sekerka flow through the hidden boundary of two merging ellipses.
The convergence of the scheme as $h\to0$ and $N\to\infty$ remains open.

The remainder of the paper is organized as follows.
\cref{sec:oriented_varifolds} recalls oriented varifolds and defines their perimeter, and \cref{sec:estimation} defines the visibility and the estimator $\widehat P_{\sigma,\delta,N}$.
\cref{sec:mathematical-analysis} states and proves \cref{thm:quantitative-perimeter-consistency} and \cref{cor:explicit-rate-with-BB-bandwidth}.
\cref{sec:application} constructs the numerical scheme for \cref{eq:mms-otto} with $d=2$, and \cref{sec:numerical} compares the estimation error with the bound of \cref{thm:quantitative-perimeter-consistency} and presents the evolution of a non-convex closed curve, the merger of a disk and an annulus and the merger of two ellipses.
\cref{app:implementation} gives the details of the implementation, and \cref{sec:proofs} proves the rate of the empirical weight measure (\cref{prop:imported-scalar-carrier-theorem}).

\paragraph{Notation.}
The unit sphere of $\R^d$ is $S^{d-1}$, $B_r(x)$ is the open ball centred at $x\in\R^d$ with radius $r>0$, $\dist(x,A)$ the Euclidean distance from $x$ to a set $A$, $\abs{E}$ the Lebesgue measure of a set $E\subset\R^d$ (we write $\d x$ for the Lebesgue measure), $\Haus{k}$ the $k$-dimensional Hausdorff measure and $\nu\llcorner A$ the restriction of a measure $\nu$ to $A$.
We write $\delta_z$ for the Dirac mass at $z$ and $\abs{\nu}$ for the total variation measure of a (possibly vector-valued) measure $\nu$.
If $\lambda$ is a measure on a space $X$ and $x\mapsto\kappa_x$ is a measurable family of finite measures on a space $Y$, $\lambda\otimes\kappa_x$, also written $\kappa_x(\d y)\lambda(\d x)$, denotes the measure on $X\times Y$ acting as $\varphi\mapsto\int_X\int_Y\varphi(x,y)\d\kappa_x(y)\d\lambda(x)$.
The space of compactly supported $C^1$ functions is $C_c^1(X;\R^m)$, $\norm{f}_\infty$ is the supremum norm (of the Euclidean norm for vector-valued $f$), $\Lip(f)$ the Lipschitz constant, $\BL(K)$ the space of bounded Lipschitz functions on a compact set $K$ with the norm $\norm{f}_{\textup{BL}(K)}\coloneqq\norm{f}_{L^\infty(K)}+\Lip(f)$ and $C^{k,\alpha}$ the H\"older spaces.
Unless otherwise stated, $C$ denotes a positive constant that depends only on the varifold $V$, the kernel $\mlfr{}$, the sampling density $\theta$ of \cref{ass:sampling-density}, the truncation level $\tau$ and the dimension $d$, in particular not on $\sigma$, $\delta$ and $N$, and whose value may change from line to line.
A unit vector in $S^{d-1}$ is denoted by $n$, as in $V(\d x,\d n)$ for an oriented varifold $V$ in \cref{def:oriented_varifolds} below.

\section{Oriented varifolds and the perimeter functional}
\label{sec:oriented_varifolds}

Two features of the one-phase Mullins--Sekerka flow guide the choice of the framework.
First, the interface is the boundary between the phase and its complement, so that a normal direction carries the phase information.
Second, once several pieces of the interface may lie on top of each other or the interface is approximated by a point cloud, oppositely oriented contributions may coexist at the same spatial location.
We therefore work with \emph{oriented varifolds}, which are measures on $\Omega\times S^{d-1}$, with the unit sphere $S^{d-1}$ in place of the Grassmannian of unoriented hyperplanes of $\R^d$.

\begin{definition}[Oriented varifold]\label{def:oriented_varifolds}
    A finite non-negative Radon measure $V$ on $\Omega\times S^{d-1}$ is called an oriented varifold on $\Omega$.
\end{definition}
For an oriented varifold $V$, one can define associated measures on $\Omega$.

\begin{definition}[Weight measure and current]
Let $V$ be an oriented varifold. Then the weight measure $\measure$ is the finite non-negative Radon measure on $\Omega$ defined by
\begin{equation}
\label{eq:intro-weight-measure}
\measure \coloneqq (\pi_\Omega)_\# V,
\end{equation}
with the projection $\pi_\Omega(x,n)\coloneqq x$, that is, $\measure(A)=V(A\times S^{d-1})$ for every Borel set $A\subset\Omega$.
Likewise, the current $\measureVector$ associated with $V$ is the finite $\R^d$-valued Radon measure on $\Omega$ defined by
\begin{equation}
\label{eq:intro-first-moment-current}
\begin{aligned}
\measureVector(\varphi)
\coloneqq
\int_{\Omega \times S^{d-1}} \varphi(x) \cdot n \d V(x,n)
&&
\text{for every compactly supported } \varphi \in C(\Omega;\R^d).
\end{aligned}
\end{equation}
\end{definition}

\begin{remark}[Relation with Hutchinson's oriented integral varifolds]
\label{rem:hutchinson}
The preceding definition is a measure-theoretic version of the current associated with an oriented varifold by \citet{Hutchinson86}.
In the rectifiable hypersurface case, an oriented integral varifold may be written, up to the identification between oriented tangent planes and unit normals, in the form
\(V = \Haus{d-1}\llcorner\Gamma \otimes \pab{ m_+(x)\delta_{n(x)} + m_-(x)\delta_{-n(x)} }\), where $\Gamma$ is a $(d-1)$-rectifiable set with unit normal $n$ and \(m_\pm(x)\in\Z_{\ge0}\) are the multiplicities of the two orientations.
Then
\(\measure = (m_+ + m_-)\Haus{d-1}\llcorner \Gamma\) and \(\measureVector = (m_+ - m_-)n\Haus{d-1}\llcorner \Gamma\).
The weight measure thus has the total multiplicity $m_+ + m_-$, whereas the current has the signed excess $m_+ - m_-$.
\end{remark}

We recall from \citet[Definition 3.35]{Ambrosio_Fusco_Pallara00} the definition of the perimeter of a set.
\begin{definition}[Perimeter functional]
    Let $E\subset\Omega$ be measurable.
    Then the perimeter of $E$ in $\Omega$ is defined by
    \[
\Per(E;\Omega) \coloneqq \sup\Set{\int_E \Div \varphi(x) \d x | \varphi \in C_c^1(\Omega;\R^d),\norm{\varphi}_\infty \le 1}.
    \]
    We say that $E$ has finite perimeter in $\Omega$ if $\Per(E;\Omega)<\infty$, and then $\Per(E;\Omega)=\abs{D\chi_E}(\Omega)$, where $D\chi_E$ is the distributional derivative of $\chi_E$.
\end{definition}
We now define a corresponding perimeter functional for oriented varifolds.

\begin{definition}[Perimeter functional for oriented varifolds]
    \label{def:perimeter-for-oriented-varifolds}
    Let $V$ be an oriented varifold. Then the perimeter of $V$ in $\Omega$ is defined as the total variation of $\measureVector$ in $\Omega$, namely
    \[
P(V;\Omega) \coloneqq \sup\Set{\measureVector(\varphi) | \varphi \in C_c^1(\Omega;\R^d),\norm{\varphi}_\infty \le 1} = \abs{\measureVector}(\Omega).
    \]
\end{definition}

\begin{remark}[Connection with finite-perimeter sets]
\label{rem:perimeter}
Let $E\subset\Omega$ be a set of finite perimeter with reduced boundary $\partial^\ast E$ and measure-theoretic outer unit normal $n_E$, so that $D\chi_E=-n_E\Haus{d-1}\llcorner\partial^\ast E$ \citep[Theorems 3.36 and 3.59]{Ambrosio_Fusco_Pallara00}.
The oriented varifold $V_E\coloneqq\Haus{d-1}\llcorner\partial^\ast E\otimes\delta_{n_E}$ satisfies $\measureVector=-D\chi_E$, the Gauss--Green measure of $E$ in the terminology of \citet[Remark 12.2]{Maggi12}, and hence $P(V_E;\Omega)=\abs{D\chi_E}(\Omega)=\Per(E;\Omega)$.
In this sense, the current $\measureVector$ extends the Gauss--Green measure from sets to oriented varifolds.
\end{remark}

\begin{definition}[First moment density]

Let $V$ be an oriented varifold on $\Omega$.
By the disintegration theorem \citep[Theorem 2.28]{Ambrosio_Fusco_Pallara00}, there exists a measurable family $\Set{\kappa_x | x\in \Omega}$ of probability measures on $S^{d-1}$ such that
\(
V(\d x,\d n)=\kappa_x(\d n)\measure(\d x).
\)
Using this family of measures, the first moment density is defined by
\(
\fmd(x) \coloneqq \int_{S^{d-1}} n \d \kappa_x(n).
\)
\end{definition}
Inserting this disintegration into \cref{eq:intro-first-moment-current} and using that $\kappa_x$ is a probability measure on $S^{d-1}$, we obtain
\begin{equation}
\label{eq:T_V-is-mV-lambdaV}
\begin{aligned}
\measureVector = \fmd\measure,
&&
\abs{\fmd(x)} \le 1
\text{ for } \measure\text{-a.e. } x \in \Omega.
\end{aligned}
\end{equation}

\begin{definition}[Hidden boundary]
\label{def:hidden-boundary}
Let $V$ be an oriented varifold.
The \emph{hidden boundary} of $V$ is the part $\measure - \abs{\measureVector} = (1-\abs{\fmd})\measure$ of the weight measure that is cancelled in the current.
The set $\Set{x | \abs{\fmd(x)}<1}$, on which it is concentrated, is also called the hidden boundary.
\end{definition}

\begin{remark}[Hidden boundaries and the measures $\kappa_x$]
\label{rem:hidden-boundary-samples-allowed}
The measure $\kappa_x$ may charge both orientations at the same point.
For example, if
\(
\kappa_x = \frac12 \delta_{n(x)} + \frac12 \delta_{-n(x)}\),
then $\fmd(x)=0$.
The weight measure then charges $x$ while the current vanishes there, and $x$ belongs to the hidden boundary.
This is the way in which two coincident pieces of the interface with opposite orientations are represented in this paper, with the weight measure $\measure$ counted once (see \cref{ass:piecewise-regular-support}).
For the varifold of \cref{rem:hutchinson}, the hidden boundary is $2\min\Bab{m_+,m_-}\Haus{d-1}\llcorner\Gamma$.
\end{remark}

\section{Estimation of the perimeter of oriented varifolds}
\label{sec:estimation}
In this section, we introduce the visible perimeter functional, which gives the perimeter term in the scheme of \cref{sec:application}, and its estimator from an oriented point cloud. We also state the assumptions on the support and on the first moment density under which the estimator is analysed in \cref{sec:mathematical-analysis}.
The visible perimeter compares the current with the weight measure after mollification at a scale $\sigma$, and approximates the perimeter $P(V;\Omega)$ of \cref{def:perimeter-for-oriented-varifolds}.

\subsection{Visible perimeter functional}

First, we introduce the kernel $\mlfr{}$ used for local averaging.
\begin{assumption}[Kernel]
\label{ass:kernel}
Let $\mlfr{}\in C_c^1(\R^d;[0,\infty))$ be radial with $\int_{\R^d}\mlfr{}(x)\d x=1$, and assume that there exist constants $0<r_0<R_0$ and $0<c_{\mlfr{}}\le C_{\mlfr{}}$ with
\begin{equation}
\label{eq:kernel-comparison}
\begin{aligned}
c_{\mlfr{}} \characteristics{B_{r_0}(0)} \le \mlfr{}\le C_{\mlfr{}} \characteristics{B_{R_0}(0)}
&&
\text{on } \R^d.
\end{aligned}
\end{equation}
For $\sigma>0$, the mollifier with scale $\sigma$ is
\begin{equation}
\label{eq:kernel-rescaling}
\mlfr{\sigma}(x) \coloneqq \frac{1}{\sigma^d} \mlfr{}\pab{\frac{x}{\sigma}},
\end{equation}
so that $\int_{\R^d}\mlfr{\sigma}(x)\d x=1$.
\end{assumption}

\begin{remark}
\label{rem:wendland}
In this paper, we choose the kernel $\mlfr{}$ to be the Wendland kernel \citep{Wendland95} defined by
$\mlfr{}(x) \coloneqq c_d\max\Bab{0,1-\abs{x}}^4(4\abs{x}+1)$ for $x\in\R^d$, where $c_d>0$ is chosen so that $\int_{\R^d}\mlfr{}(x)\d x=1$.
\end{remark}

We now define the visible perimeter functional.

\begin{definition}[Visible perimeter functional]
    \label{def:coherent-perimeter-functional}
For $\sigma > 0$, the weight measure and the current are mollified by convolution with the kernel $\mlfr{\sigma}$,
\[
\begin{aligned}
    (\mlfr{\sigma} \ast \measure)(x) \coloneqq \int_{\R^d}\mlfr{\sigma}(x-y)\d \measure(y),
&&
(\mlfr{\sigma} \ast \measureVector)(x) \coloneqq \int_{\R^d}\mlfr{\sigma}(x-y)\d\measureVector(y),
\end{aligned}
\]
and the \emph{visibility} is defined by their ratio:
\[
q_\sigma[V](x)
\coloneqq
\begin{cases}
\dfrac{\abs{(\mlfr{\sigma} \ast \measureVector)(x)}}{(\mlfr{\sigma} \ast \measure)(x)},&\text{if }(\mlfr{\sigma} \ast \measure)(x) > 0,\\
0&\text{otherwise}.
\end{cases}
\]
Since $(\mlfr{\sigma} \ast \measureVector)(x)/(\mlfr{\sigma} \ast \measure)(x)$ is the $\mlfr{\sigma}(x-\cdot)\d V$-weighted average of the orientations $n$, the visibility is the length of the mean of the local unit normals (for fixed $x$, the mean resultant length of directional statistics, see \citealp{MardiaJupp00}). It equals $1$ where all normals near $x$ agree, and it vanishes where the mollified current vanishes, for instance where two pieces of the interface with opposite normals cancel exactly within the support of $\mlfr{\sigma}(x-\cdot)$ (a hidden boundary in the sense of \cref{def:hidden-boundary}).
The \emph{visible perimeter} functional is
\[
P_\sigma(V) \coloneqq \int_\Omega q_\sigma[V](x) \d \measure(x).
\]
\end{definition}
Since $\abs{\mlfr{\sigma}\ast\measureVector}\le\mlfr{\sigma}\ast\abs{\measureVector}\le\mlfr{\sigma}\ast\measure$ by \cref{eq:T_V-is-mV-lambdaV}, one has $0 \le q_\sigma[V](x) \le 1$.

\subsection{Geometric support assumptions}
Below, we impose geometric assumptions on the support $\Gamma \coloneqq \supp{\measure}$ of the weight measure and on the first moment density $\fmd$.
The assumptions on $\Gamma$ and on the sampling density $\theta$ of \cref{ass:sampling-density} are close to the regularity class of \citet[Section 6]{Boricaud_Buet25}.

\begin{assumption}[Piecewise regular support]
\label{ass:piecewise-regular-support}
There exist finitely many compact sets $\Gamma_1,\dots,\Gamma_L \subset \Omega$, which we call \emph{branches}, and a set $\Sigma \subset \Omega$ which is a finite union of compact $C^1$ submanifolds of dimension $d-2$ (a finite set if $d=2$), such that
\(
\Gamma = \bigcup_{l=1}^L \Gamma_l\) and
\(\Gamma_l\cap\Gamma_{l'} \subset \Sigma\) for \(l\ne l'\).
Moreover,
\begin{enumerate}[label=\textup{(\roman*)},ref=(\roman*)]
\item each $\Gamma_l$ is an embedded $(d-1)$-dimensional $C^{1,\alpha}$ manifold with (possibly empty) boundary $\partial\Gamma_l\subset\Sigma$, for a fixed exponent $\alpha \in (0,1]$,
\item\label{item:weight-measure-hausdorff} the weight measure is
\(\measure = \Haus{d-1} \llcorner \Gamma\),
\item\label{item:branch-separation} there exist $c_{\textup{s}}\in(0,1]$ and $r_{\textup{sep}}>0$ such that, with $d_\Sigma(y)\coloneqq\min\Bab{\dist(y,\Sigma),r_{\textup{sep}}}$ for $y\in\R^d$ ($d_\Sigma\equiv r_{\textup{sep}}$ if $\Sigma=\emptyset$),
\begin{equation}
\label{eq:branch-separation-global}
\begin{aligned}
\dist(x,\Gamma_{l'}) \ge c_{\textup{s}}d_\Sigma(x)
&&
\text{for every } x\in\Gamma_l \text{ and every } l'\ne l.
\end{aligned}
\end{equation}
\end{enumerate}
\Cref{item:branch-separation} rules out distinct branches meeting tangentially at $\Sigma$. Since the branches are compact and meet only in $\Sigma$, it suffices to require \cref{eq:branch-separation-global}, with a possibly smaller $c_{\textup{s}}$, for $x$ with $\dist(x,\Sigma)\le r_{\textup{sep}}$. For instance, it holds there if $\Gamma_l\pitchfork\Gamma_{l'}$ for all $l\ne l'$.
When $\Sigma=\emptyset$ (for instance, when $\Gamma$ is a finite union of disjoint closed surfaces), it only states that distinct branches are at positive distance from each other.
\Cref{item:weight-measure-hausdorff} fixes the multiplicity of $\Gamma$ to $1$, and hidden boundaries are carried by $\kappa_x$ as in \cref{rem:hidden-boundary-samples-allowed}.
\end{assumption}

The first moment density is allowed to be only piecewise H\"older continuous.

\begin{assumption}[Piecewise H\"older first moment density]
\label{ass:piecewise-holder-mV}
There exists an exponent $\beta \in (0,1]$ such that for each $l\in\{1,\dots,L\}$ there is a map $\fmdbranch{l}\in C^{0,\beta}(\Gamma_l;\R^d)$ with $\fmd=\fmdbranch{l}$ $\measure$-almost everywhere on $\Gamma_l$.
Since $\measure(\Sigma)=0$ (\cref{lem:ahlfors-regularity} below) and the sets $\Gamma_l\setminus\Sigma$ are pairwise disjoint, we fix the representative $\fmd\coloneqq\fmdbranch{l}$ on $\Gamma_l\setminus\Sigma$, so that $\fmd$ is defined at every point of $\Gamma\setminus\Sigma$.
\end{assumption}

\begin{lemma}[Ahlfors regularity and singular neighbourhoods]
\label{lem:ahlfors-regularity}
Under \cref{ass:piecewise-regular-support}, there exist constants $c_\Gamma, C_\Gamma, r_\Gamma > 0$ such that
\begin{equation}
\label{eq:ahlfors-regularity}
c_\Gamma r^{d-1} \le \measure(B_r(x)) \le C_\Gamma r^{d-1}
\end{equation}
for every $x \in \Gamma$ and every $r \in (0,r_\Gamma]$.
This property is called the Ahlfors regularity of $\measure$.
Moreover, there exists a constant $C_\Sigma > 0$ such that, for every $r \in (0,r_\Gamma]$,
\begin{equation}
\label{eq:singular-neighbourhood-measure}
\measure\pab{\Set{ x \in \Gamma | \dist(x,\Sigma) \le r }} \le C_\Sigma r.
\end{equation}
In particular, $\measure(\Sigma)=0$.
\end{lemma}

\begin{proof}
A compact $C^1$ manifold $\Gamma_l$ with boundary has $(d-1)$-dimensional density bounded above and below at small scales, uniformly in the point (at a boundary point the ball $B_r(x)$ still contains a half-ball portion of $\Gamma_l$), and since there are finitely many manifolds the constants can be chosen uniformly in $l$.
A ball $B_r(x)$ centred at $x\in\Gamma$ meets at most $L$ of the manifolds, so that the upper bound holds for $\Gamma$ after changing the constant, and the lower bound holds because $x$ belongs to at least one $\Gamma_l$.
This proves \cref{eq:ahlfors-regularity}.
For \cref{eq:singular-neighbourhood-measure}, the $r$-neighbourhood of a compact $C^1$ submanifold of dimension $d-2$ is covered by $Cr^{-(d-2)}$ balls of radius $2r$. A ball of radius $2r$ that meets $\Gamma$ is contained in a ball of radius $4r$ centred at a point of $\Gamma$, which carries $\measure$-mass at most $C_\Gamma(4r)^{d-1}$ by \cref{eq:ahlfors-regularity} if $4r\le r_\Gamma$. Hence the total mass is at most $4^{d-1}CC_\Gamma r$ for $r\le r_\Gamma/4$, and for $r_\Gamma/4<r\le r_\Gamma$ the bound holds once $C_\Sigma\ge4\measure(\Gamma)/r_\Gamma$.
For $d=2$, $\Sigma$ is a finite set and the bound follows directly from \cref{eq:ahlfors-regularity}.
\end{proof}

We fix
\begin{equation}
\label{eq:sigma-zero}
\sigma_0\coloneqq\min\Bab{1,\ \frac{r_\Gamma}{R_0+1},\ \frac{c_{\textup{s}}\min\{r_\Gamma,r_{\textup{sep}}\}}{R_0}}
\end{equation}
with $r_\Gamma$ as in \cref{lem:ahlfors-regularity} and $c_{\textup{s}}$, $r_{\textup{sep}}$ as in \cref{ass:piecewise-regular-support}.
Since $R_0\sigma\le r_\Gamma$ for $\sigma\le\sigma_0$, \cref{eq:kernel-comparison,eq:ahlfors-regularity} give, with $c_\mu\coloneqq c_{\mlfr{}}c_\Gamma r_0^{d-1}$,
\begin{equation}
\label{eq:U-sigma-comparable}
\begin{aligned}
(\mlfr{\sigma}\ast\measure)(x)
\ge
\frac{c_{\mlfr{}}}{\sigma^d} \measure(B_{r_0\sigma}(x))
\ge
c_\mu \sigma^{-1}
&&
\text{for every } x \in \Gamma \text{ and every } \sigma \in (0,\sigma_0].
\end{aligned}
\end{equation}

\subsection{Sampling model and empirical visible perimeter}
We introduce the sampling model of \citet{Boricaud_Buet25}, weights that correct for non-uniform sampling, and the estimator of the visible perimeter.
First, we assume that the positions of the samples have a density $\theta$ with respect to $\measure$ and impose a regularity assumption on $\theta$.
\begin{assumption}[Sampling density]
\label{ass:sampling-density}
There exist a function $\theta$ on $\Gamma$ and $0 < \theta_- \le \theta_+ < \infty$ such that
\begin{equation}
\label{eq:sampling-density-bounds}
\begin{aligned}
\theta_- \le \theta(x) \le \theta_+
&&
\text{for } \Haus{d-1}\text{-a.e. } x \in \Gamma,
&&
\int_\Gamma\theta\d\measure=1,
\end{aligned}
\end{equation}
and we write $\mu_X\coloneqq\theta\measure$ for the probability measure with density $\theta$.
Assume moreover that $\theta$ is piecewise $C^{0,\gamma}$ on the decomposition in \cref{ass:piecewise-regular-support}, that is, there exists $\gamma\in(0,1]$ such that for each $l\in\{1,\dots,L\}$ there is $\theta_l\in C^{0,\gamma}(\Gamma_l)$ with $\theta=\theta_l$ $\measure$-almost everywhere on $\Gamma_l$.
As in \cref{ass:piecewise-holder-mV}, we fix the representative $\theta\coloneqq\theta_l$ on $\Gamma_l\setminus\Sigma$.
\end{assumption}

We now state the assumption on the law of the oriented samples $(X_i,U_i)$.

\begin{assumption}[i.i.d.~oriented samples]
\label{ass:marked-iid-model}
Let $(X_i,U_i)_{i=1}^N$ be independent and identically distributed (i.i.d.) random variables, defined on a probability space with probability measure $\PP$, with values in $\Omega \times S^{d-1}$ and with common law
\(
\mu_{X,U}(\d x,\d n)
=
\theta(x)\kappa_x(\d n)\measure(\d x)\),
whose first marginal is $\mu_X$.
Throughout, $\EE$ denotes the expectation with respect to $\PP$ and $\ExpectationSample{\cdot}$ (resp.\ $\ExpectationGiven{\cdot}{X_i=x}$) the conditional expectation given the positions (resp.\ given $X_i=x$).
In particular,
\begin{equation}
\label{eq:conditional-mean-of-marks}
\begin{aligned}
\ExpectationGiven{U_i}{X_i=x} = \fmd(x)
&&
\text{for } \mu_X\text{-almost every } x \in \Gamma.
\end{aligned}
\end{equation}
\end{assumption}

From the sample $(X_i,U_i)$ and from weights $W_i^{\delta,N}$, we introduce an empirical weight measure and an empirical current, which correspond to the weight measure $\measure$ and to the current $\measureVector$ of \cref{eq:intro-weight-measure,eq:intro-first-moment-current}.

\begin{definition}[Empirical weight measure and current]
\label{def:empirical-mass-and-current}
    The empirical weight measure is the finite non-negative Radon measure on $\Omega$ defined by
\begin{equation}
\label{eq:corrected-carrier-measure}
\measureMassEmpirical{\delta,N}
\coloneqq
\frac{1}{N} \sum_{i=1}^N W_i^{\delta,N} \delta_{X_i},
\end{equation}
where $W_i^{\delta,N}\ge0$ are weights to be defined in \cref{eq:weights-from-BB}.
In addition, define the empirical current by
\begin{equation}
\label{eq:corrected-current-measure}
\measureVectorEmpirical{\delta,N}
\coloneqq
\frac{1}{N} \sum_{i=1}^N W_i^{\delta,N} U_i \delta_{X_i}.
\end{equation}
\end{definition}

Before stating \cref{prop:imported-scalar-carrier-theorem}, we introduce the bounded-Lipschitz distance, the metric used by \citet{Boricaud_Buet25}.
\begin{definition}[Bounded-Lipschitz distance]
\label{def:bounded-lipschitz-distance}
Let $K \subset \overline\Omega$ be compact.
For finite signed Radon measures $\nu_1,\nu_2$ on $K$, define
\begin{equation}
\label{eq:bounded-lipschitz-distance}
d_{\textup{BL},K}(\nu_1,\nu_2)
\coloneqq
\sup\Set{ \int_K \varphi \d (\nu_1-\nu_2) | \varphi \in \BL(K),\ \norm{\varphi}_{\textup{BL}(K)} \le 1 }.
\end{equation}
\end{definition}
If $\nu_1,\nu_2$ are supported in a compact set $K\subset K'\subset\overline\Omega$, then $d_{\textup{BL},K}(\nu_1,\nu_2)\le d_{\textup{BL},K'}(\nu_1,\nu_2)$, because every $\varphi\in\BL(K)$ extends to $K'$ with the same Lipschitz constant by the McShane extension theorem \citep{McShane1934} and, after truncation to $[-\norm{\varphi}_{L^\infty(K)},\norm{\varphi}_{L^\infty(K)}]$, with the same norm.

We now define the weights $W_i^{\delta,N}$, following \citet[Section 4]{Boricaud_Buet25}.
For a bandwidth $\delta>0$, their density estimator is
\[
\begin{aligned}
    \theta_{\delta,N}(x) \coloneqq \frac{1}{N Z_{\mlfr{}}^{(d-1)}\delta^{d-1}} \sum_{i=1}^N \mlfr{}\pab{\frac{x-X_i}{\delta}},
    &&
    Z_{\mlfr{}}^{(d-1)} \coloneqq \int_{\R^{d-1}}\mlfr{}(x',0)\d x',
\end{aligned}
\]
that is, the uncorrected empirical weight measure $\frac{1}{N}\sum_{i=1}^N \delta_{X_i}$ is convolved with the kernel rescaled with respect to the \emph{intrinsic} dimension $d-1$ of the support, as in \citet[Section 3.2]{Boricaud_Buet25} (since $\mlfr{}$ is radial, $Z_{\mlfr{}}^{(d-1)}$ is the integral of $\mlfr{}$ over any hyperplane through the origin). This rescaling differs from the ambient one $\mlfr{\delta} = \delta^{-d}\mlfr{}(\cdot/\delta)$ of \cref{eq:kernel-rescaling}, which is used for the visibility.
For a truncation level $\tau\in(0,\theta_-]$, let $\Phi_\tau(t) \coloneqq \cutoff{\tau}(t) / t$ for $t>0$ and $\Phi_\tau(0)\coloneqq0$, where
\(
\cutoff{\tau}(t) \coloneqq \min\Bab{1,\max\Bab{0,\nicefrac{2t}{\tau}-1}}
\)
vanishes for $t\le\tau/2$, equals $1$ for $t\ge\tau$ and is linear in between, and define the corrected weights by
\begin{equation}
\label{eq:weights-from-BB}
\begin{aligned}
W_i^{\delta,N} \coloneqq \Phi_\tau(\theta_{\delta,N}(X_i)),
&&
i=1,\dots,N.
\end{aligned}
\end{equation}
Since $\cutoff{\tau}(t)/t = 2/\tau - 1/t \le 1/\tau$ for $t\in[\tau/2,\tau)$ and $\cutoff{\tau}(t)/t = 1/t \le 1/\tau$ for $t \ge \tau$, we have
\begin{equation}
\label{eq:weights-bounded}
\begin{aligned}
0 \le W_i^{\delta,N} \le \frac{1}{\tau}
&&
\text{for every } i\in \{1,\dots,N\}
\end{aligned}
\end{equation}
(see also \citealp[Lemma 4.7 (i)]{Boricaud_Buet25}). These weights are the ones used in the computations of \cref{sec:application,sec:numerical}.
We measure the accuracy of the empirical weight measure by
\begin{equation}
\label{eq:BB-carrier-estimate}
\varepsilon(\delta,N)
\coloneqq
\Expectation{d_{\textup{BL},\overline\Omega}(\measureMassEmpirical{\delta,N},\measure)},
\end{equation}
which also depends on $V$, $\theta$ and $\tau$.
\begin{proposition}[Rate of the empirical weight measure]
\label{prop:imported-scalar-carrier-theorem}
Under \cref{ass:kernel,ass:piecewise-regular-support,ass:sampling-density,ass:marked-iid-model}, let the weights be given by \cref{eq:weights-from-BB} with a fixed $\tau\in(0,\theta_-]$, write $\alpha\wedge\gamma\coloneqq\min\{\alpha,\gamma\}$ and set
\begin{equation}
\label{eq:BB-bandwidth-choice}
\begin{aligned}
\delta_N \coloneqq N^{-\nicefrac{1}{(d-1+2(\alpha\wedge\gamma))}},
&&
r \coloneqq \frac{\alpha\wedge\gamma}{d-1+2(\alpha\wedge\gamma)}.
\end{aligned}
\end{equation}
Then there exists a constant $C_0>0$ such that
\begin{equation}
\label{eq:BB-rate}
\begin{aligned}
\varepsilon(\delta_N,N)
\le
C_0 N^{-r}
&&
\text{for every } N\in\Z_{>0}.
\end{aligned}
\end{equation}
\end{proposition}
The proof is given in \cref{sec:proofs}.

\begin{remark}
The rate \cref{eq:BB-rate} is that of \citet[Theorem 7.4]{Boricaud_Buet25}, with the H\"older exponents $a=\alpha$, $b=\gamma$ of that paper and with its intrinsic dimension replaced by $d-1$.
That theorem concerns a variant in which the density $\theta_{\delta,N}$ is computed from a sample independent of the one carrying the weights.
For the estimator \cref{eq:corrected-carrier-measure} with a single sample, \citet[Theorem 4.9]{Boricaud_Buet25} bound the distance to a mollified measure, which together with their Proposition 6.12 gives a slower rate, and \cref{prop:imported-scalar-carrier-theorem} shows that the rate of Theorem 7.4 also holds for it.
\end{remark}

We now define the estimator of the visible perimeter, a discrete analogue of \cref{def:coherent-perimeter-functional} built from the same sample.

\begin{definition}[Empirical visible perimeter]
\label{def:empirical-coherent-perimeter-functional}
For $\sigma > 0$ and $x\in\Omega$, mollifying the empirical weight measure and current as in \cref{def:coherent-perimeter-functional} gives
\[
\begin{aligned}
\pab{\mlfr{\sigma} \ast \measureMassEmpirical{\delta,N}}(x)
= \frac{1}{N}\sum_{i=1}^N W_i^{\delta,N} \mlfr{\sigma}(x-X_i),
&&
\pab{\mlfr{\sigma} \ast \measureVectorEmpirical{\delta,N}}(x)
= \frac{1}{N}\sum_{i=1}^N W_i^{\delta,N} U_i \mlfr{\sigma}(x-X_i).
\end{aligned}
\]
Then we define the empirical visibility by
\begin{equation}
\label{eq:discrete-visibility-density}
q_{\sigma,\delta,N}(x)
\coloneqq
\begin{dcases}
\frac{\abs{\pab{\mlfr{\sigma} \ast \measureVectorEmpirical{\delta,N}}(x)}}{\pab{\mlfr{\sigma} \ast \measureMassEmpirical{\delta,N}}(x)},& \text{if } \pab{\mlfr{\sigma} \ast \measureMassEmpirical{\delta,N}}(x) > 0,\\
0 & \text{otherwise}.
\end{dcases}
\end{equation}
The empirical visible perimeter is
\begin{equation}
\label{eq:discrete-perimeter-estimator}
\widehat P_{\sigma,\delta,N}
\coloneqq
\int_\Omega q_{\sigma,\delta,N}(x) \d \measureMassEmpirical{\delta,N}(x)
=
\frac{1}{N} \sum_{i=1}^N W_i^{\delta,N}
\frac{\abs{\pab{\mlfr{\sigma} \ast \measureVectorEmpirical{\delta,N}}(X_i)}}{\pab{\mlfr{\sigma} \ast \measureMassEmpirical{\delta,N}}(X_i)}.
\end{equation}
Since $\mlfr{\sigma}(0) > 0$ and $W_i^{\delta,N} \ge 0$, every denominator in \cref{eq:discrete-perimeter-estimator} is strictly positive whenever $W_i^{\delta,N} > 0$.
\end{definition}

\section{Convergence of the empirical visible perimeter}
\label{sec:mathematical-analysis}

We establish a quantitative consistency result between the perimeter $P(V;\Omega)$ of \cref{def:perimeter-for-oriented-varifolds} and the empirical visible perimeter $\widehat P_{\sigma,\delta,N}$ as $\sigma\to 0$, $\delta\to 0$ and $N\to \infty$, under the assumptions of \cref{sec:estimation}.

\begin{theorem}[Quantitative perimeter consistency]
\label{thm:quantitative-perimeter-consistency}
Under \cref{ass:kernel,ass:piecewise-regular-support,ass:piecewise-holder-mV,ass:sampling-density,ass:marked-iid-model}, there exists a constant $C > 0$ such that, for every $\sigma \in (0,\sigma_0]$ (with $\sigma_0$ as in \cref{eq:sigma-zero}), every $\delta > 0$, and every $N \in \Z_{>0}$,
\begin{equation}
\label{eq:quantitative-perimeter-bound}
\Expectation{\abs{\widehat P_{\sigma,\delta,N} - P(V;\Omega)}}
\le
C \pab{
\sigma^\beta
+
\sigma^{-d} \varepsilon(\delta,N)
+
\frac{1}{N\sigma^{d-1}}
+
\frac{1}{\sqrt{N\sigma^{d-1}}}
}.
\end{equation}
\end{theorem}
For suitable sequences $\delta_N,\sigma_N\downarrow0$, this gives an explicit rate of decay of the expected error.

\begin{corollary}
\label{cor:explicit-rate-with-BB-bandwidth}
Under the assumptions of \cref{thm:quantitative-perimeter-consistency}, let $\delta_N$ and $r$ be as in \cref{eq:BB-bandwidth-choice}, and set $\sigma_N\coloneqq\Lambda N^{-\nicefrac{r}{(d+\beta)}}$ with a constant $\Lambda\ge1$.
Then there exists a constant $C > 0$ such that, for every $N\in\Z_{>0}$,
\begin{equation}
    \label{eq:coherent-final-estimate-N}
\Expectation{\abs{\widehat P_{\sigma_N,\delta_N,N} - P(V;\Omega)}}
\le
C N^{-\nicefrac{\beta r}{(d+\beta)}}.
\end{equation}
\end{corollary}

The proof of \cref{thm:quantitative-perimeter-consistency} rests on two lemmas.
The first one is deterministic. Away from $\Sigma$, the kernel $\mlfr{\sigma}(x-\cdot)$ meets one branch only, on which $\fmd$ is H\"older continuous, so that the mollification of a measure with the density $\fmd$ is close to $\fmd(x)$ times the mollification of the measure.
The second one controls the fluctuation of the sampled orientations $U_i$ about their conditional means $\fmd(X_i)$.
The theorem and the corollary are proved at the end of this section.

\begin{lemma}[Comparison away from $\Sigma$]
\label{lem:continuous-quantitative-consistency}
Under \cref{ass:kernel,ass:piecewise-regular-support,ass:piecewise-holder-mV}, let $\sigma \in (0,\sigma_0]$ with $\sigma_0$ as in \cref{eq:sigma-zero}, let $\overline R \coloneqq R_0/c_{\textup{s}}$, and set
\(
\Gamma_\sigma^{\textup{sing}} \coloneqq \Set{ x \in \Gamma | \dist(x,\Sigma) \le \overline R \sigma }\),
\(\Gamma_\sigma^{\textup{reg}} \coloneqq \Gamma \setminus \Gamma_\sigma^{\textup{sing}}\).
Then there exists a constant $C > 0$ such that the following holds.
\begin{enumerate}[label=\textup{(\roman*)},ref=(\roman*)]
\item\label{item:comparison-general-measure} For every $x\in\Gamma_\sigma^{\textup{reg}}$ and every finite non-negative Radon measure $\nu$ on $\Gamma$ with $\nu(\Sigma)=0$,
\begin{equation}
\label{eq:holder-control-inside-kernel}
\abs{(\mlfr{\sigma}\ast(\fmd\nu))(x) - \fmd(x) (\mlfr{\sigma}\ast\nu)(x)}
\le
C \sigma^\beta (\mlfr{\sigma}\ast\nu)(x).
\end{equation}
\item\label{item:comparison-visibility} The visibility satisfies $\abs{q_\sigma[V](x) - \abs{\fmd(x)}} \le C \sigma^\beta$ for every $x \in \Gamma_\sigma^{\textup{reg}}$, and the singular part satisfies $\measure(\Gamma_\sigma^{\textup{sing}}) \le C \sigma$.
\item\label{item:comparison-perimeter} The visible perimeter satisfies \(\abs{P_\sigma(V) - P(V;\Omega)} \le C \sigma^\beta\).
\end{enumerate}
\end{lemma}

\begin{proof}
Let $x \in \Gamma_\sigma^{\textup{reg}}$ belong to the branch $\Gamma_l$.
Then $d_\Sigma(x)=\min\{\dist(x,\Sigma),r_{\textup{sep}}\}\ge\min\{\overline R\sigma,r_{\textup{sep}}\}= R_0\sigma/c_{\textup{s}}$, because $\sigma\le c_{\textup{s}}r_{\textup{sep}}/R_0$ by \cref{eq:sigma-zero}. Hence $\dist(x,\Gamma_{l'})\ge c_{\textup{s}}d_\Sigma(x)\ge R_0\sigma$ for every $l'\ne l$ by \cref{eq:branch-separation-global}, and $\dist(x,\Sigma)>R_0\sigma$. The ball $B_{R_0\sigma}(x)$, outside of which $\mlfr{\sigma}(x-\cdot)$ vanishes, therefore meets $\Gamma$ only in $\Gamma_l\setminus\Sigma$, where $\fmd=\fmdbranch{l }$ and $\abs{\fmd(y)-\fmd(x)}\le C\abs{y-x}^\beta\le C\sigma^\beta$.
Since $\nu\ge0$, this gives
\[
\abs{(\mlfr{\sigma}\ast(\fmd\nu))(x) - \fmd(x) (\mlfr{\sigma}\ast\nu)(x)}
=
\abs{\int_\Gamma \mlfr{\sigma}(x-y) ( \fmd(y)-\fmd(x) ) \d \nu(y)}
\le
C \sigma^\beta (\mlfr{\sigma}\ast\nu)(x),
\]
which is \ref{item:comparison-general-measure}.

For \ref{item:comparison-visibility}, we take $\nu=\measure$, which satisfies $\measure(\Sigma)=0$ by \cref{lem:ahlfors-regularity} and $\fmd\measure=\measureVector$ by \cref{eq:T_V-is-mV-lambdaV}. Dividing \cref{eq:holder-control-inside-kernel} by $(\mlfr{\sigma}\ast\measure)(x)$, which is positive by \cref{eq:U-sigma-comparable}, and using $\abs{\abs{a}-\abs{b}}\le \abs{a - b}$ for $a,b\in\R^d$, we obtain the first claim. The second one follows from \cref{eq:singular-neighbourhood-measure}, since $\overline R\sigma\le r_\Gamma$ by \cref{eq:sigma-zero}.

Finally, $P(V;\Omega)=\abs{\measureVector}(\Omega)=\int_\Gamma\abs{\fmd}\d\measure$ by \cref{eq:T_V-is-mV-lambdaV}, and $0\le q_\sigma[V],\abs{\fmd}\le1$, so that \ref{item:comparison-visibility} gives
\[
\abs{P_\sigma(V) - P(V;\Omega)}
\le
\int_{\Gamma} \abs{q_\sigma[V] - \abs{\fmd}} \d \measure
\le
C \sigma^\beta\measure(\Gamma) + \measure(\Gamma_\sigma^{\textup{sing}})
\le
C \sigma^\beta,
\]
where we used $\sigma\le1$ and $\beta\le1$.
\end{proof}

\begin{lemma}[Fluctuation of the orientations]
\label{lem:mark-fluctuation-bound}
Under \cref{ass:kernel,ass:piecewise-regular-support,ass:sampling-density,ass:marked-iid-model}, with the weights \cref{eq:weights-from-BB}, define
\(
M_{\sigma,\delta,N}(x)
\coloneqq
(\mlfr{\sigma} \ast (\measureVectorEmpirical{\delta,N} - \fmd\measureMassEmpirical{\delta,N}))(x)\).
Then there exists a constant $C > 0$ such that
\begin{equation}
\label{eq:mark-fluctuation-average-bound}
\Expectation{
\frac{1}{N} \sum_{i=1}^N W_i^{\delta,N} \abs{M_{\sigma,\delta,N}(X_i)}
}
\le
C\pab{\frac{1}{N\sigma^d} + \frac{1}{\sqrt{N\sigma^{d+1}}}}
\end{equation}
holds for every $\sigma \in (0,\sigma_0]$, every $\delta>0$ and every $N\in\Z_{>0}$, with $\sigma_0$ as in \cref{eq:sigma-zero}.
\end{lemma}

\begin{proof}
By definition, $M_{\sigma,\delta,N}(x)=\frac1N\sum_{j=1}^NW_j^{\delta,N}\mlfr{\sigma}(x-X_j)\pab{U_j-\fmd(X_j)}$.
Given the positions $X_1,\dots,X_N$, the weights are deterministic, the orientations $U_j$ are independent, and the differences $U_j-\fmd(X_j)$ are centred by \cref{eq:conditional-mean-of-marks} and bounded by $2$.
Hence the cross terms vanish in the conditional second moment, and \cref{eq:weights-bounded} gives, for every $i\in\{1,\dots,N\}$,
\[
\begin{aligned}
\ExpectationSample{\abs{M_{\sigma,\delta,N}(X_i)}^2}
&=
\frac{1}{N^2}\sum_{j=1}^N\pab{W_j^{\delta,N}}^2\mlfr{\sigma}(X_i-X_j)^2\,\ExpectationSample{\abs{U_j-\fmd(X_j)}^2}\\
&\le
\frac{4}{\tau^2N^2}\sum_{j=1}^N\mlfr{\sigma}(X_i-X_j)^2.
\end{aligned}
\]
By \cref{eq:weights-bounded}, the exchangeability of the samples and the Cauchy--Schwarz inequality,
\[
\begin{aligned}
\Expectation{\frac{1}{N} \sum_{i=1}^N W_i^{\delta,N} \abs{M_{\sigma,\delta,N}(X_i)}}
&\le
\frac{1}{\tau}\pab{\Expectation{\abs{M_{\sigma,\delta,N}(X_1)}^2}}^{\nicefrac12}\\
&\le
\frac{2}{\tau^2N}\pab{\mlfr{\sigma}(0)^2+(N-1)\Expectation{\mlfr{\sigma}(X_1-X_2)^2}}^{\nicefrac12}.
\end{aligned}
\]
Here $\mlfr{\sigma}(0)^2\le C\sigma^{-2d}$ by \cref{eq:kernel-rescaling}, and \cref{eq:sampling-density-bounds,eq:ahlfors-regularity} (applicable since $R_0\sigma\le r_\Gamma$) give
\(
\ExpectationGiven{\mlfr{\sigma}(X_1-X_2)^2}{X_1}
\le
\theta_+\norm{\mlfr{}}_\infty^2\sigma^{-2d}\measure(B_{R_0\sigma}(X_1))
\le
C\sigma^{-(d+1)}
\).
This proves \cref{eq:mark-fluctuation-average-bound}.
\end{proof}

\begin{proof}[Proof of \cref{thm:quantitative-perimeter-consistency}]
Fix $\sigma \in (0,\sigma_0]$, $\delta>0$ and $N\in\Z_{>0}$, and write
\[
\begin{aligned}
u_N &\coloneqq \mlfr{\sigma} \ast \measureMassEmpirical{\delta,N}, &
j_N &\coloneqq \mlfr{\sigma} \ast \measureVectorEmpirical{\delta,N}, &
q_N &\coloneqq q_{\sigma,\delta,N},\\
u &\coloneqq \mlfr{\sigma} \ast \measure, &
j &\coloneqq \mlfr{\sigma} \ast \measureVector, &
q &\coloneqq q_\sigma[V].
\end{aligned}
\]
Since $\widehat P_{\sigma,\delta,N}=\int_\Omega q_N \d\measureMassEmpirical{\delta,N}$ and $P_\sigma(V)=\int_\Omega q\d\measure$, we have
\begin{equation}
\label{eq:error-decomposition}
\widehat P_{\sigma,\delta,N} - P(V;\Omega)
=
\int_\Omega (q_N-q) \d \measureMassEmpirical{\delta,N}
+
\int_\Omega q \d (\measureMassEmpirical{\delta,N}-\measure)
+
\pab{P_\sigma(V)-P(V;\Omega)}.
\end{equation}
The last term is at most $C\sigma^\beta$ by \cref{lem:continuous-quantitative-consistency}. We estimate the expectations of the first two terms.

\paragraph{The first term.}
Almost surely, the samples belong to $\Gamma\setminus\Sigma$, since $\mu_X\ll\measure$ and $\measure(\Sigma)=0$ by \cref{lem:ahlfors-regularity}, so that $\measureMassEmpirical{\delta,N}$ is a measure on $\Gamma$ with $\measureMassEmpirical{\delta,N}(\Sigma)=0$.
Let $x=X_i$ with $W_i^{\delta,N}>0$, so that $u_N(x)>0$ and $q_N(x)=\abs{j_N(x)}/u_N(x)$.
If $x\in\Gamma_\sigma^{\textup{reg}}$, then $j_N=M_{\sigma,\delta,N}+\mlfr{\sigma}\ast(\fmd\measureMassEmpirical{\delta,N})$ and \cref{eq:holder-control-inside-kernel} with $\nu=\measureMassEmpirical{\delta,N}$ give $\abs{\abs{j_N(x)}-\abs{\fmd(x)}u_N(x)}\le\abs{M_{\sigma,\delta,N}(x)}+C\sigma^\beta u_N(x)$. Dividing by $u_N(x)$ and using $0\le q_N,\abs{\fmd}\le1$, we get
\[
\abs{q_N(x)-\abs{\fmd(x)}}
\le
\min\Bab{1,\frac{\abs{M_{\sigma,\delta,N}(x)}}{u_N(x)}}+C\sigma^\beta.
\]
The minimum is at most $2\pab{\abs{M_{\sigma,\delta,N}(x)}+\abs{u_N(x)-u(x)}}/u(x)$. Indeed, if $u_N(x)\ge u(x)/2$, this quantity bounds $\abs{M_{\sigma,\delta,N}(x)}/u_N(x)$, and otherwise $2\abs{u_N(x)-u(x)}/u(x)>1$.
With $u(x)\ge c_\mu\sigma^{-1}$ from \cref{eq:U-sigma-comparable} and \cref{lem:continuous-quantitative-consistency}\ref{item:comparison-visibility}, we obtain
\[
\abs{q_N(x)-q(x)}
\le
\frac{2\sigma}{c_\mu}\pab{\abs{M_{\sigma,\delta,N}(x)} + \abs{u_N(x)-u(x)}}+C\sigma^\beta
\]
for $x=X_i\in\Gamma_\sigma^{\textup{reg}}$.
If $x\in\Gamma_\sigma^{\textup{sing}}$, we use $\abs{q_N(x)-q(x)}\le1$.
Since $\measureMassEmpirical{\delta,N}(\Omega)\le1/\tau$ by \cref{eq:weights-bounded}, it follows that
\begin{equation}
\label{eq:first-term-pointwise}
\abs{\int_\Omega (q_N-q) \d \measureMassEmpirical{\delta,N}}
\le
\frac{2\sigma}{c_\mu}\pab{\frac{1}{N}\sum_{i=1}^N W_i^{\delta,N}\abs{M_{\sigma,\delta,N}(X_i)}+\frac{1}{\tau}\norm{u_N-u}_{L^\infty(\Omega)}}
+C\sigma^\beta
+\measureMassEmpirical{\delta,N}(\Gamma_\sigma^{\textup{sing}}).
\end{equation}
The expectation of the sum is bounded by \cref{lem:mark-fluctuation-bound}.
For $x\in\Omega$, the function $\varphi_x\coloneqq\mlfr{\sigma}(x-\cdot)$ satisfies $\norm{\varphi_x}_{\textup{BL}(\Gamma)}\le(\norm{\mlfr{}}_\infty+\norm{\nabla\mlfr{}}_\infty)\sigma^{-(d+1)}$, so that \cref{eq:bounded-lipschitz-distance} gives
\[
\abs{u_N(x)-u(x)}
=
\abs{\int_\Gamma\varphi_x\d(\measureMassEmpirical{\delta,N}-\measure)}
\le
C\sigma^{-(d+1)}d_{\textup{BL},\Gamma}(\measureMassEmpirical{\delta,N},\measure).
\]
By the remark after \cref{def:bounded-lipschitz-distance}, $d_{\textup{BL},\Gamma}\le d_{\textup{BL},\overline\Omega}$, and \cref{eq:BB-carrier-estimate} then gives $\Expectation{\norm{u_N-u}_{L^\infty(\Omega)}}\le C\sigma^{-(d+1)}\varepsilon(\delta,N)$.
Finally, \cref{eq:weights-bounded}, \cref{eq:sampling-density-bounds} and \cref{lem:continuous-quantitative-consistency}\ref{item:comparison-visibility} give
\[
\Expectation{\measureMassEmpirical{\delta,N}(\Gamma_\sigma^{\textup{sing}})}
=\Expectation{W_1^{\delta,N}\characteristics{\Gamma_\sigma^{\textup{sing}}}(X_1)}
\le \frac{\theta_+}{\tau}\measure(\Gamma_\sigma^{\textup{sing}})\le C\sigma.
\]
Taking the expectation in \cref{eq:first-term-pointwise} and using $\sigma\le\sigma^\beta$, we obtain
\begin{equation}
\label{eq:first-term-final}
\Expectation{\abs{\int_\Omega (q_N-q) \d \measureMassEmpirical{\delta,N}}}
\le
C\pab{\sigma^\beta+\sigma^{-d}\varepsilon(\delta,N) + \frac{1}{N\sigma^{d-1}}+ \frac{1}{\sqrt{N\sigma^{d-1}}}}.
\end{equation}

\paragraph{The second term.}
Both $\measureMassEmpirical{\delta,N}$ and $\measure$ are supported in $\Gamma$, so only the restriction $q|_\Gamma$ matters.
We claim that $q|_\Gamma$ is Lipschitz with $\Lip(q|_\Gamma)\le C\sigma^{-1}$.
Indeed, let $x,y\in\Gamma$.
If $\abs{x-y}\ge\sigma$, then $\abs{q(x)-q(y)}\le 1\le \sigma^{-1}\abs{x-y}$.
If $\abs{x-y}<\sigma$, then, by \cref{eq:kernel-rescaling,eq:ahlfors-regularity} (applicable since $(R_0+1)\sigma\le r_\Gamma$ by \cref{eq:sigma-zero}),
\[
\abs{j(x)-j(y)} + \abs{u(x)-u(y)}
\le
2\norm{\nabla\mlfr{}}_\infty\sigma^{-(d+1)}\abs{x-y}\measure\pab{B_{(R_0+1)\sigma}(x)}
\le
C\sigma^{-2}\abs{x-y},
\]
while $u(x),u(y)\ge c_\mu\sigma^{-1}$ by \cref{eq:U-sigma-comparable} and $\abs{j}\le u$ on $\Gamma$ by \cref{eq:T_V-is-mV-lambdaV}. Writing $q(x)-q(y) = \nicefrac{\abs{j(x)}-\abs{j(y)}}{u(x)} + \abs{j(y)}\cdot(\nicefrac{u(y)-u(x)}{u(x)u(y)})$ then gives $\abs{q(x)-q(y)}\le C\sigma^{-1}\abs{x-y}$.
By the McShane extension theorem \citep{McShane1934}, $q|_\Gamma$ extends to a function $\widetilde q$ on $\overline\Omega$ with the same Lipschitz constant and, after truncation to $[0,1]$, with $\norm{\widetilde q}_{\textup{BL}(\overline{\Omega})}\le 1 + C\sigma^{-1} \le C\sigma^{-1}$ since $\sigma \le 1$.
Since $\int_\Omega q \d (\measureMassEmpirical{\delta,N}-\measure) = \int_\Omega \widetilde q \d (\measureMassEmpirical{\delta,N}-\measure)$, we deduce that
\begin{equation}
\label{eq:measure-difference-term}
\Expectation{\abs{\int_\Omega q \d (\measureMassEmpirical{\delta,N}-\measure)}}
\le
\norm{\widetilde q}_{\textup{BL}(\overline\Omega)}\Expectation{d_{\textup{BL},\overline\Omega}(\measureMassEmpirical{\delta,N}, \measure)}
\le
C \sigma^{-1} \varepsilon(\delta,N)
\end{equation}
by \cref{eq:BB-carrier-estimate}.
Finally, inserting \cref{eq:first-term-final}, \cref{eq:measure-difference-term} and \cref{lem:continuous-quantitative-consistency}\ref{item:comparison-perimeter} into \cref{eq:error-decomposition} and using $\sigma^{-1}\le\sigma^{-d}$ gives \cref{eq:quantitative-perimeter-bound}.
\end{proof}

\begin{proof}[Proof of \cref{cor:explicit-rate-with-BB-bandwidth}]
In this proof, the constants $C$ may also depend on $\Lambda$ and on the constant $C_0$ of \cref{eq:BB-rate}, and we write $a\coloneqq r/(d+\beta)$, so that $\sigma_N=\Lambda N^{-a}$.
Let $N_0$ be the smallest integer with $N_0\ge(\Lambda/\sigma_0)^{\nicefrac{1}{a}}$, so that $\sigma_N\le\sigma_0$ for $N\ge N_0$.
For such $N$, \cref{eq:BB-rate} gives $\varepsilon(\delta_N,N) \le C_0 N^{-r}$, and the terms of the right-hand side of \cref{eq:quantitative-perimeter-bound} satisfy
\begin{gather}
\sigma_N^\beta \le C N^{-\beta a},
\quad
\sigma_N^{-d} \varepsilon(\delta_N,N) \le C N^{da - r}= C N^{-\beta a},\\
\frac{1}{N\sigma_N^{d-1}}
\le
\frac{1}{\sqrt{N\sigma_N^{d-1}}} \le C N^{-\nicefrac{(1-(d-1)a)}{2}}\le C N^{-\beta a}.
\end{gather}
Here $da-r=-\beta a$ by the definition of $a$, and the last chain uses $N\sigma_N^{d-1}\ge1$ and $(d-1+2\beta)a\le1$. Both follow from $(d+1)a\le1$, which holds because $a\le r\le1/(d+1)$ by $\alpha\wedge\gamma\le1$.
This proves \cref{eq:coherent-final-estimate-N} for $N\ge N_0$.
For the finitely many indices $N<N_0$, we use $0\le q_{\sigma,\delta,N}\le1$, $0\le\abs{\fmd}\le1$ and \cref{eq:weights-bounded}, which give $\Expectation{\abs{\widehat P_{\sigma_N,\delta_N,N} - P(V;\Omega)}}\le 1/\tau+\measure(\Omega)$.

The exponent $a$ balances the first two terms. For $\sigma_N=\Lambda N^{-a'}$ with $a'\ne a$, the larger of the exponents $-\beta a'$ and $da'-r$ is larger than $-\beta a$.
\end{proof}

\section{A point-cloud minimizing movement scheme}
\label{sec:application}

This section explains how the empirical visible perimeter introduced in \cref{def:empirical-coherent-perimeter-functional} is used as the perimeter term in a boundary point-cloud approximation of Otto's minimizing movement scheme for the one-phase Mullins--Sekerka flow \citep{Otto98,ChambolleLaux21}, and how the resulting scheme passes through a change of topology.
We consider the case $d=2$ below and describe the scheme for a phase bounded by a single closed curve. A phase with several boundary curves or several connected components is treated in the same way, curve by curve and component by component, as described in \cref{app:implementation} and in \cref{subsec:application-events}.
\Cref{thm:quantitative-perimeter-consistency} motivates the perimeter term in the scheme, whereas the convergence of the fully discrete scheme is an open problem.
The scheme described below is the one used for all computations of \cref{sec:numerical}.

\subsection{Continuous problem}

We recall the minimizing movement scheme \cref{eq:mms-otto} of \citet{Otto98}, in which $E^{n-1}\subset\Omega\subset\R^2$ is the previous phase and $h>0$ the time step.
For two non-negative measures $\nu_1$, $\nu_2$ on $\Omega$ with the same total mass, the Wasserstein distance is $W_2^2(\nu_1,\nu_2)\coloneqq\inf_{\pi\in\Pi(\nu_1,\nu_2)}\int_{\Omega\times\Omega}\abs{x-y}^2\d\pi(x,y)$ with the set $\Pi(\nu_1,\nu_2)$ of couplings of $\nu_1$ and $\nu_2$, and $W_2\coloneqq+\infty$ if the total masses differ.

Two structural facts about the one-phase flow guide the discretization.
The normal velocity is $v=-\partial_{n} u$, where $u$ is the harmonic function in $E$ with boundary values $u=\kappa$, the curvature of $\partial E$ with the sign convention $\kappa=\Div_{\partial E}n_E$ (positive for convex $E$), and $\partial_n$ is the derivative in the direction of the outer normal $n=n_E$.
Since $u$ is determined separately on each connected component $E_j$ of $E$, disjoint components do not interact, and both the area and the barycentre of every component are conserved, that is,
\begin{equation}
\label{eq:application-conservation}
\begin{aligned}
\odv{}{t}\abs{E_j}
=
-\int_{\partial E_j}\partial_{n} u\dH{1} = 0,
&&
\odv{}{t}\int_{E_j}x\d x
=
-\int_{\partial E_j}x\partial_{n} u\dH{1}
=
-\int_{\partial E_j}\kappa n\dH{1}
=
0,
\end{aligned}
\end{equation}
where the second identity uses $\Delta u=0$ and the last one the closedness of $\partial E_j$.
The discrete scheme enforces both conservation laws at the linearized level for each connected component of the phase (\cref{subsec:application-normal-graph}). This keeps the evolution of two separate components before contact, and of the merged component after contact, consistent with the conserved quantities of the flow.

\subsection{Discrete state and masses}
\label{subsec:application-discrete-state}

At time step $n$, we represent the boundary $\partial E$ by an oriented point-cloud varifold
\[
\begin{aligned}
V_N^n
\coloneqq
\frac{1}{N}\sum_{i=1}^N W_i^n\delta_{(x_i^n,u_i^n)},
&&
u_i^n=\pab{\cos\alpha_i^n,\sin\alpha_i^n},
\end{aligned}
\]
where $u_i^n$ is the stored outer unit normal and $W_i^n$ is the corrected weight of \cref{eq:weights-from-BB}. We write
\(w_i^n\coloneqq \frac{1}{N}W_i^n\)
for the particle mass and use only $w_i^n$ from now on, so that the empirical weight measure and current of \cref{eq:corrected-carrier-measure,eq:corrected-current-measure} read $\sum_i w_i^n\delta_{x_i^n}$ and $\sum_i w_i^n u_i^n\delta_{x_i^n}$.
The kernel $\mlfr{}$ is the Wendland kernel of \cref{rem:wendland}. The density bandwidth $\delta$ and the truncation level $\tau$ are chosen from the initial point cloud by a nearest-neighbour rule and kept fixed, and $\sigma$ is fixed as well (\cref{app:implementation}).

\subsection{Normal-graph variables and conservation constraints}
\label{subsec:application-normal-graph}

We parametrize the positions $\Bab{x_i^n}$ over the previous point cloud by normal displacements,
\begin{equation}
\label{eq:application-position-update}
\begin{aligned}
x_i(s)=x_i^{n-1}+s_i u_i^{n-1},
&&
i=1,\dots,N.
\end{aligned}
\end{equation}
The normal angle $\alpha_i$ is coupled to $s$ by a weak discretization of the relation $\delta\alpha=-\odv{s}{\varsigma}$ between the rotation of the normal and the derivative of a normal graph with respect to the arclength $\varsigma$.
With the masses $w_i^{n-1}$, the tangents $t_i^{n-1}\coloneqq\pab{-(u_i^{n-1})_2,(u_i^{n-1})_1}$ and the kernel $\mlfr{\sigma}$ as test functions, define $A^{n-1},B^{n-1}\in\R^{N\times N}$ by
\begin{equation}
\label{eq:application-angle-matrices}
\begin{aligned}
A_{ki}^{n-1}
\coloneqq
w_i^{n-1}
\mlfr{\sigma}\pab{x_i^{n-1}-x_k^{n-1}},
&&
B_{ki}^{n-1}
\coloneqq
w_i^{n-1}
t_i^{n-1}\cdot
\nabla\mlfr{\sigma}\pab{x_i^{n-1}-x_k^{n-1}},
\end{aligned}
\end{equation}
and set, with $(A^{n-1})^\dagger$ the Moore--Penrose pseudo-inverse,
\begin{equation}
\label{eq:application-angle-constraint}
\begin{aligned}
\delta\alpha_i
\coloneqq
\abb{(A^{n-1})^\dagger B^{n-1}s}_i,
&&
\alpha_i(s)=\alpha_i^{n-1}+\delta\alpha_i,
&&
u_i(s)=\pab{\cos\alpha_i(s),\sin\alpha_i(s)},
\end{aligned}
\end{equation}
with the usual truncated singular value decomposition when the system is ill-conditioned.

\paragraph{Conservation rows.}
The conservation laws \cref{eq:application-conservation} are imposed as linear constraints on $(s,\delta\alpha)$.
Each particle carries a segment of length $w_i^{n-1}$ with three quadrature nodes $\xi_k\in\Bab{-\tfrac12,0,\tfrac12}$ ($k=1,2,3$) of equal weight,
\begin{equation}
\label{eq:application-collocation-points}
\begin{aligned}
p_{ik}^{n-1}
\coloneqq
x_i^{n-1}
+
\xi_k w_i^{n-1}t_i^{n-1},
&&
\zeta_{ik}^{n-1}
\coloneqq
\tfrac13 w_i^{n-1},
&&
v_{ik}(s)
\coloneqq
s_i-\xi_k w_i^{n-1}\delta\alpha_i .
\end{aligned}
\end{equation}
The last quantity is the normal displacement of the node $p_{ik}$ of a rotating segment.
For a connected phase $E$, the linearized area and first moment are the flux sums
\begin{equation}
\label{eq:application-conservation-rows}
\begin{aligned}
\delta A[s]
\coloneqq
\sum_{i=1}^N\sum_{k}\zeta_{ik}^{n-1}v_{ik}(s),
&&
\delta M[s]
\coloneqq
\sum_{i=1}^N\sum_{k}\zeta_{ik}^{n-1}p_{ik}^{n-1}v_{ik}(s),
\end{aligned}
\end{equation}
which discretize the transport identities $\odv{}{t}\abs{E}=\int_{\partial E}v\dH{1}$ and $\odv{}{t}\int_{E}x\d x=\int_{\partial E}xv\dH{1}$ without any ordering of the boundary points.
The \emph{conservation rows} are $\delta A[s]=0$ and $\delta M[s]-\overline x\delta A[s]=0$, where $\overline x$ is the barycentre of $E$, computed from the boundary integrals $\abs{E}=\frac12\int_{\partial E}x\cdot n_E\dH1$ and $\int_{E}x\d x=\frac12\int_{\partial E}\abs{x}^2n_E\dH1$ discretized on the segments as in \cref{eq:application-conservation-rows} (the centring makes the rows invariant under translations of the origin).
These rows are the constraints of the minimization in \cref{subsec:application-discrete-step}.
When the phase has several connected components, the rows are imposed for each component (\cref{app:implementation}).
Collecting the rows in a matrix $C^{n-1}$, we write $s=Q^{n-1}y$ where the columns of $Q^{n-1}$ form an orthonormal basis of $\ker C^{n-1}$.
The rows are quadratures of the transport identities with the masses and normals of the previous state. The area $A$ of \cref{sec:numerical}, computed from the updated positions, masses and normals, is a different discrete quantity, which they do not conserve exactly, and its drift is reported in \cref{sec:numerical}.

\subsection{Linearized Wasserstein term by a boundary integral method}
\label{subsec:application-linearized-wasserstein}

The Wasserstein term in \cref{eq:mms-otto} is replaced by its tangent quadratic form at the previous phase.
We recall its formal derivation, following the dynamic formulation of optimal transport \citep{BenamouBrenier2000} and Otto's interpretation of the one-phase Mullins--Sekerka flow as a Wasserstein gradient flow \citep{Otto98}.
Let $\partial E$ be smooth and let $v$ be a normal velocity on $\partial E$ with $\int_{\partial E_j} v\dH{1}=0$ on every connected component $E_j$ of $E$.
In Otto's formal Riemannian structure, the squared length of $v$ as a tangent vector at the phase $E$ is $\int_E\abs{\nabla\phi_v}^2\d x$, where $\phi_v$ solves the Neumann problem $\Delta\phi_v=0$ in $E$, $\partial_{n}\phi_v=v$ on $\partial E$. This integral is the minimal kinetic energy of the divergence-free velocity fields in $E$ whose normal component on $\partial E$ is $v$ \citep{BenamouBrenier2000}.
If $E_s$ denotes the phase bounded by $\Set{x+s(x)n_E(x) | x\in\partial E}$, the boundary displaced by $s$ in the normal direction during a time step $h$, we set $v=s/h$ and formally linearize
\begin{equation}
\label{eq:application-wasserstein-linearization}
\frac{1}{2h}W_2^2\pab{\chi_{E_s}\d x,\chi_E\d x}
\approx
\frac{h}{2}\int_E\abs{\nabla\phi_v}^2\d x.
\end{equation}

The scheme evaluates the right-hand side of \cref{eq:application-wasserstein-linearization} on the point cloud itself by a boundary integral method, which we describe for a connected phase $E$ with smooth boundary.
The potential $\phi_v$ is the single-layer potential of a density $\lambda$ on $\partial E$,
\[
\begin{aligned}
\phi_v(x)=\int_{\partial E}G(x-y)\lambda(y)\dH1(y),
&&
G(z)\coloneqq-\frac1{2\pi}\log\abs{z},
&&
\pab{\frac12I+K^\ast}\lambda=v
\quad\text{on }\partial E,
\end{aligned}
\]
where $K^\ast$ is the adjoint double-layer operator, $(K^\ast\lambda)(x)=\int_{\partial E}\partial_{n(x)}G(x-y)\lambda(y)\dH1(y)$. This integral equation, which we call the Neumann equation, is solvable exactly when $\int_{\partial E}v\dH1=0$, its solution is unique up to the kernel of $\tfrac12I+K^\ast$, which is spanned by the density $\lambda_{\textup{eq}}$ whose single-layer potential is constant on $E$ (adding it to $\lambda$ adds a constant to $\phi_v$), and $\int_E\abs{\nabla\phi_v}^2\d x=\int_{\partial E}\phi_vv\dH1$ by Green's identity.
The discretization uses the nodes $p_{ik}^{n-1}$ and weights $\zeta_{ik}^{n-1}$ of \cref{eq:application-collocation-points} as collocation points and quadrature weights, with $\abs{z}^2$ replaced by $\abs{z}^2+\varepsilon_{\textup{reg}}^2$ in $G$ and $\partial_nG$, $\varepsilon_{\textup{reg}}\coloneqq0.15\ell$, where $\ell$ is the median of the masses $w_i^{n-1}$.
Let $S$ and $K^\ast_N$ be the resulting $3N\times3N$ matrices (columns scaled by the quadrature weights), and let $\widetilde K\coloneqq D^{\nicefrac12}\pab{\tfrac12I+K^\ast_N}D^{-\nicefrac12}$ with $D$ the diagonal matrix of the quadrature weights.
The matrix $\widetilde K$ has one singular value close to zero. Its right singular vector $a_{\textup{R}}$ is the discrete counterpart of $D^{\nicefrac12}\lambda_{\textup{eq}}$ and its left singular vector $a_{\textup{L}}$ that of the constants, that is, of the compatibility condition.
The Neumann equation with the data $V(s)\coloneqq\pab{v_{ik}(s)}_{i,k}$ is solved by the system bordered with $a_{\textup{L}}$ and $a_{\textup{R}}$ (\cref{app:implementation}), which is invertible. Its bordering unknown is $\mu=a_{\textup{L}}\transpose D^{\nicefrac12}V(s)$, the component of the data along the discrete compatibility direction, and the system solves the equation with the data projected onto the complement of $a_{\textup{L}}$.
The area row of \cref{eq:application-conservation-rows} is the quadrature of the continuous compatibility condition $\int_{\partial E}v\dH1=0$ and differs from $a_{\textup{L}}\transpose D^{\nicefrac12}V(s)=0$ by the discretization. The unknown $\mu$ is monitored as a residual (\cref{app:implementation}).
The linearized Wasserstein term is
\begin{equation}
\label{eq:application-bie-wasserstein}
\widehat W^{n-1}_{h}(s)
\coloneqq
\frac1{2h}V(s)\transpose DS\Lambda V(s),
\end{equation}
where $\Lambda$ is the solution map $V\mapsto\lambda$ of the bordered system, so that $S\Lambda V(s)/h$ is the discrete potential of the velocity $V(s)/h$, and $\widehat W^{n-1}_h(s)$ the discrete form of $\frac h2\int_{\partial E}\phi_vv\dH1$.
We call the resulting quadratic form the \emph{metric} of the scheme.
The matrix $DS\Lambda$ is assembled once per step (one singular value decomposition and one factorization of order $3N$). Each evaluation of $\widehat W^{n-1}_{h}$ and of its derivatives then costs $\order{N^2}$ operations, and no grid enters.
When the phase has several connected components, the construction is applied to each component and the forms are added. Separate components then do not interact, as for the continuous flow, and the area row of each component is the quadrature of its compatibility condition (\cref{app:implementation}).

\subsection{One fully discrete step}
\label{subsec:application-discrete-step}

The full visible perimeter $\widehat P_{\sigma,\delta,N}$ in \cref{eq:discrete-perimeter-estimator} depends on the positions through both the weights and the visibility.
In the minimizing movement step, we use the semi-implicit perimeter, with the visibility of the previous step,
\[
\widehat P^{n-1}(s)
\coloneqq
\sum_{i=1}^N w_i(s)q_i^{n-1},
\]
where $w_i(s)$ is the particle mass of \cref{subsec:application-discrete-state} with the weight \cref{eq:weights-from-BB} evaluated at the displaced positions $x_i(s)$, so that $w_i(0)=w_i^{n-1}$, and $q_i^{n-1}$ is the visibility \cref{eq:discrete-visibility-density} of $V_N^{n-1}$ at $x_i^{n-1}$ (with several curves it is modified as in \cref{eq:application-step-visibility} or \cref{eq:application-cc-visibility}).
Taking the visibility from the previous step prevents the minimization from reducing the perimeter term by rotating nearby normals into a configuration of low visibility within a single step.
The discrete step minimizes
\begin{equation}
\label{eq:application-discrete-mm-step}
\begin{aligned}
F^{n-1}(y)
\coloneqq
\widehat P^{n-1}\pab{Q^{n-1}y}
+
\widehat W^{n-1}_{h}\pab{Q^{n-1}y},
&&
y\in\R^{N-\rank C^{n-1}},
\end{aligned}
\end{equation}
and $y^n$ is the point returned by the trust-region Newton method of \cref{app:implementation}, started at $y=0$. The returned point is stationary for $F^{n-1}$ up to a gradient tolerance.
We then set $s^n=Q^{n-1}y^n$ and update positions and normals by \cref{eq:application-position-update,eq:application-angle-constraint}, which gives $V_N^n$ after the redistribution of \cref{subsec:application-tangential-redistribution} and the update of the masses.
We write $\widehat P^{n}\coloneqq\widehat P^{n-1}(s^n)$ for the value of the perimeter term at $y^n$. This is the quantity reported in \cref{sec:numerical}, and it differs from the visible perimeter \cref{eq:discrete-perimeter-estimator} of the new state $V_N^n$ by the update of the visibility and by the redistribution of \cref{subsec:application-tangential-redistribution}.

\subsection{Tangential redistribution}
\label{subsec:application-tangential-redistribution}

After each minimizing movement step, the point distribution may become non-uniform even when the geometric interface remains well behaved.
Motivated by the artificial tangential velocity method of \citet{PanDongGuoShi2026}, we redistribute the particles along the tangent direction by a tangential diffusion of the point density.
With a redistribution bandwidth $\delta_{\textup{red}}>0$, let $\theta^{\textup{red}}\coloneqq\theta_{\delta_{\textup{red}},N}$ be the density estimator of \cref{sec:estimation} with $d=2$, computed from the current positions $x_j$, and move each particle along its tangent,
\[
\widetilde x_i
\coloneqq
x_i-h_{\textup{red}}\pab{\nabla\log \theta^{\textup{red}}(x_i)\cdot t_i}t_i,
\]
where $t_i$ is the tangent associated with the stored normal $u_i$ as in \cref{subsec:application-normal-graph} and $h_{\textup{red}}>0$ is a step size.
Particles move from regions of high sampling density towards regions of low sampling density, only in the tangent direction.
The stored normal is transported along the curve by the first-order correction $\widetilde\alpha_i\coloneqq\alpha_i+\kappa_i^{\textup{red}}(\widetilde x_i-x_i)\cdot t_i$, where $\kappa_i^{\textup{red}}$ is the regularized mean curvature estimator of \citet{BuetLeonardiMasnou17} (see also \citealp{BuetRumpf22}) evaluated with the masses $w_j$.
The update is iterated until the coefficient of variation of $\theta^{\textup{red}}$ along each closed curve falls below a tolerance, and a substep that increases it is rejected.
This tangential update is a reparametrization of the discrete point cloud. In the smooth limit, it leaves the geometric interface unchanged, and at the discrete level it keeps the quadrature accuracy of the masses and of the collocation \cref{eq:application-collocation-points}.

\subsection{Topological events}
\label{subsec:application-events}

A change of topology involves several closed curves, which we call \emph{loops}, and is carried out by four \emph{events}, instantaneous changes of the discrete state, each subject to a criterion fixed in advance and to a check. The four events are the switch of the visibility, the merging of the constraints, the removal of the \emph{cancelling pair}, that is, of two coincident pieces of the interface with opposite normals, and the reconnection of the remaining arcs.
The treatment of several loops and the definition of the events are given in \cref{app:implementation}.

\section{Numerical experiments}\label{sec:numerical}
All computations of this section were performed with our implementation, which is publicly available.\footnote{\url{https://doi.org/10.5281/zenodo.23113244}}

\subsection{Numerical validation of the convergence rate}

We first test the convergence rate of the empirical visible perimeter \cref{eq:discrete-perimeter-estimator} against the bound of \cref{thm:quantitative-perimeter-consistency}.
The support is the unit circle $\Gamma=S^1\subset\Omega\coloneqq B_2(0)$ with the uniform density $\theta\equiv1/(2\pi)$, so that $\alpha=\gamma=1$ and $r=1/3$, and the kernel is the Wendland kernel of \cref{rem:wendland}.
Two models for the orientations are used.
For $\beta=1$, the orientations are the outer normals, $U_i=X_i$, so that $\fmd(x)=x$ and $P(V;\Omega)=2\pi$.
For $\beta=\nicefrac12$, the orientations are $U_i=\pm X_i$ with $\PP(U_i=X_i\mid X_i)=(1+a(X_i))/2$, where $a(x)\coloneqq\sqrt{\dist_{S^1}(x,e_1)/\pi}$, $e_1\coloneqq(1,0)$ and $\dist_{S^1}$ is the arclength distance on the circle. Then $\fmd(x)=a(x)x$ is $C^{0,\nicefrac12}$ with a cusp at $e_1$, the hidden boundary is carried by the measures $\kappa_x$ as in \cref{rem:hidden-boundary-samples-allowed}, and $P(V;\Omega)=\int_{S^1}a\dH1=4\pi/3$.
The scales are $\delta_N=N^{-\nicefrac13}$, $\tau=\theta_-/2$ and $\sigma_N=N^{-\nicefrac19}$ for $\beta=1$, $\sigma_N=N^{-\nicefrac{1}{12}}$ for $\beta=\nicefrac12$. For each $N\in\{10^2,10^3,10^4,\sci{3}{4},10^5,10^6\}$, the estimator is computed on $10$ independent samples, and \cref[abbrev]{fig:num-rate-validation} shows the mean of the relative error with error bars of one standard deviation.

\begin{figure}
    \centering
    \includegraphics{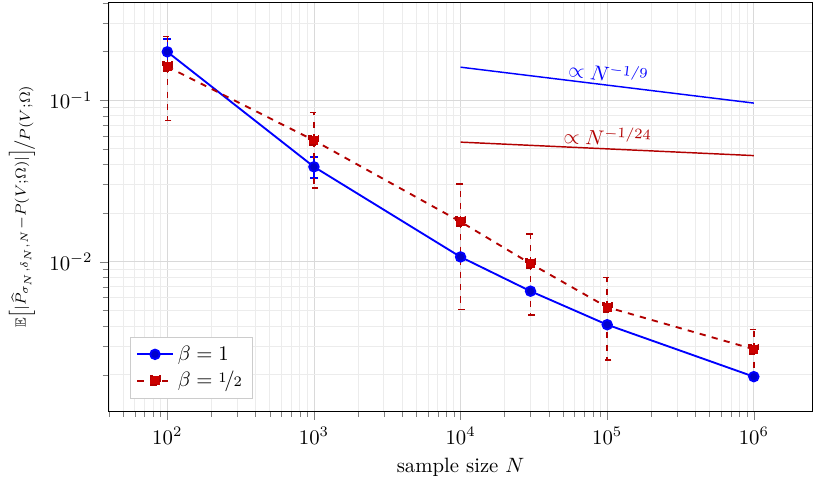}
    \caption{Relative error of the empirical visible perimeter. The lines without markers, labelled by their slopes $-\nicefrac{1}{9}$ and $-\nicefrac{1}{24}$, show the exponents obtained by evaluating the bound \cref{eq:quantitative-perimeter-bound} along the sequences $\delta_N=N^{-\nicefrac13}$ and $\sigma_N=N^{-\nicefrac19}$ ($\beta=1$), $\sigma_N=N^{-\nicefrac{1}{12}}$ ($\beta=\nicefrac12$). Along the sequences of \cref{cor:explicit-rate-with-BB-bandwidth} ($\sigma_N=N^{-\nicefrac19}$ for $\beta=1$ and $\sigma_N=N^{-\nicefrac{2}{15}}$ for $\beta=\nicefrac12$) the exponents are $\nicefrac19$ and $\nicefrac{1}{15}$.}
    \label{fig:num-rate-validation}
\end{figure}

The observed decay in \cref[abbrev]{fig:num-rate-validation} is faster than the bound \cref{eq:quantitative-perimeter-bound} predicts along the chosen sequences, for both models of the orientations.

\subsection{One-phase Mullins--Sekerka dynamics with a change of topology}

We test the scheme of \cref{sec:application} on a single non-convex loop and on two mergers, a concentric configuration whose evolution is known in closed form, and the local contact of two ellipses.
In both mergers, the point cloud passes through the change of topology by the events of \cref{subsec:application-events} in a single run.

\paragraph{Setting.}
All runs start from the point spacing $2\pi/256\simeq\sci{2.5}{-2}$ (arclength sampling), so that $\ell\simeq2\pi/256$ at the initial time, and use $\sigma=0.1$, $\delta$ and $\tau$ from \cref{eq:application-delta-tau-rule}, $h=10^{-5}$ and double precision. One step takes about one second on a workstation ($N\le500$), the assembly of the boundary integral forms included.
None of the stopping criteria of \cref{app:implementation} was met.
We report the perimeter term $\widehat P^{n}$ of \cref{subsec:application-discrete-step}, the area $A\coloneqq\frac12\sum_iw_ix_i\cdot u_i$ of each component, its barycentre $\overline x$, defined by $\sum_iw_i\abs{x_i-\overline x}^2u_i=0$, the circularity $4\pi A/(\widehat P^{n})^{2}$ and, for a single loop, the mean radius $\overline r$ about the barycentre. The sums are restricted to the loops of the component. They discretize the boundary integrals of \cref{subsec:application-normal-graph} at the particles.
The circularity is the quantity by which \citet{Julin23} measure the approach of two-dimensional flat flows to unions of equal disks.

\subsubsection{Single component: a five-petal domain}

The initial domain is $\Set{x | \abs{x}<0.75+0.25\cos(5\varphi)}$, where $\varphi$ is the polar angle of $x$. The run uses $N=291$ and takes $4000$ steps ($t\le0.04$), and \cref[abbrev]{fig:num-flower} shows the evolution.
The perimeter term $\widehat P^{n}$ decreases at every step from $7.09$ to $4.84$ and the circularity increases from $0.47$ to $1.00$. At the end, the state is stationary (maximal normal displacement $\sci{3.5}{-8}$ per step), its perimeter is within $0.2\%$ of $2\pi\sqrt{A/\pi}$, and the relative area error is $-\sci{2.3}{-3}$, two thirds of which arise during the first $50$ steps.

\begin{figure}
    \centering
    \includegraphics[width=\linewidth]{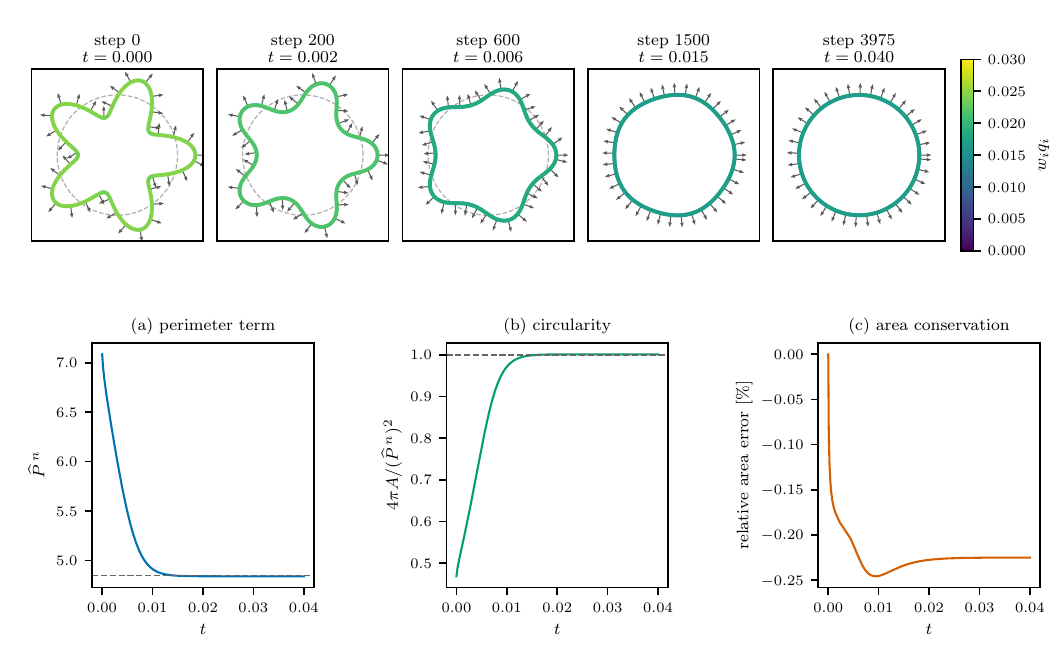}
    \caption{Five-petal domain. Top: point cloud at five times, coloured by $w_iq_i$ (the dashed circle has the area of the initial domain). Bottom: (a) perimeter term $\widehat P^{n}$ (the dashed line is $2\pi\sqrt{A/\pi}$); (b) circularity; (c) relative area error in percent.}
    \label{fig:num-flower}
\end{figure}

\subsubsection{A disk inside an annulus}

Let $E(0)=B_{R_1}(0)\cup\Set{x | R_2<\abs{x}<R_3}$ with $R_1=0.35$, $R_2=0.55$, $R_3=1$.
The disk is stationary. In the annulus $\Set{x | R_-<\abs{x}<R_+}$, the potential is $u=k\log\abs{x}+k_0$, where the constants $k$ and $k_0$ are determined by $u=\kappa$ on both circles (that is, $u=1/R_+$ on the outer circle and $u=-1/R_-$ on the inner one), and $v=-\partial_n u$ gives
\begin{equation}
\label{eq:num-annulus-ode}
\begin{aligned}
\odv{R_\pm}{t} = -\frac{k(R_+,R_-)}{R_\pm},
&&
k(R_+,R_-)=\frac{R_+^{-1}+R_-^{-1}}{\log(R_+/R_-)},
&&
(R_+,R_-)(0)=(R_3,R_2),
\end{aligned}
\end{equation}
so that $\pi(R_+^2-R_-^2)$ is conserved and the hole reaches the disk at $T_c=0.020347$ (by quadrature).
By conservation of the total area, the merged set is the disk of radius $R_\textup{eq}=R_+(T_c)=0.905539$, an equilibrium. At the contact time $T_c$, the circle $\abs{x}=R_1$ carries two coincident pieces of the interface with opposite orientations, a configuration which the scheme handles through the events of \cref{subsec:application-events} and which lies outside \cref{ass:piecewise-regular-support}.
The contact is simultaneous along the whole circle, which is the price for an exact solution.
The three circles carry $N=487$ particles, and the run takes $2400$ steps ($t\le0.024\simeq1.18T_c$).
\Cref{fig:num-annulus} shows the run, and the times $t_1<t_2$ are defined below.

\begin{figure}
    \centering
    \includegraphics[width=\linewidth]{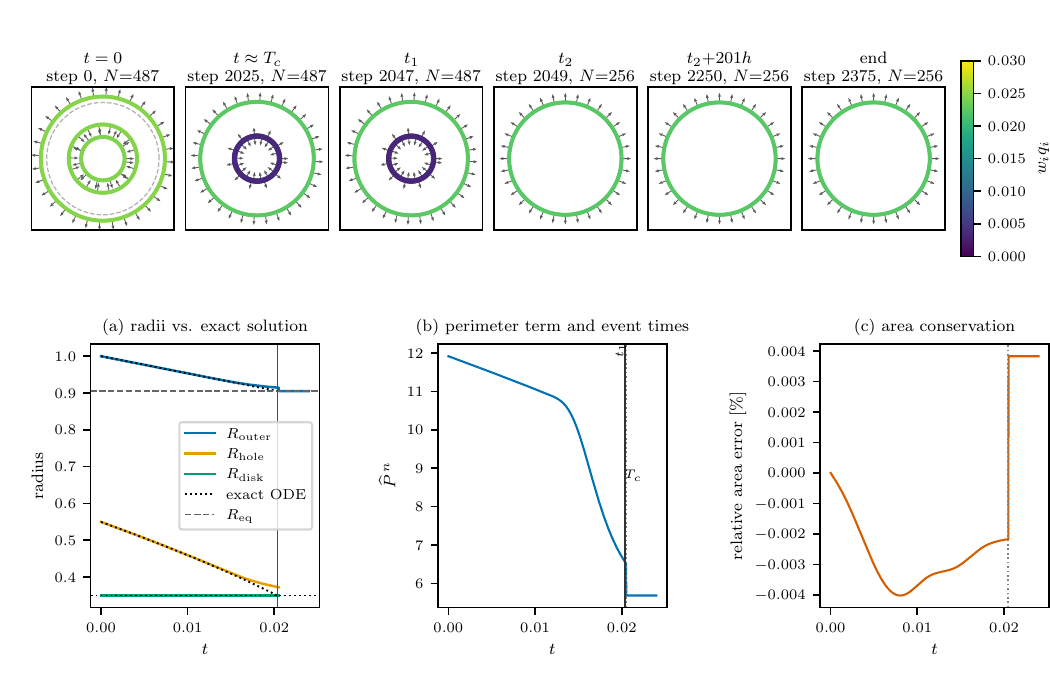}
    \caption{Top: point cloud at $t=0$, near the exact contact time $T_c$, at $t_1$ (merging of the constraints), at $t_2$ (removal of the pair), about $200$ steps later and near the end of the run, coloured by $w_iq_i$ (normals drawn every eighth point; the dashed circle has radius $R_\textup{eq}$). The cancelling pair (disk and hole) is visible at $t_1$ as the dark ring of low visibility. Bottom: (a) the three radii against the exact solution of \cref{eq:num-annulus-ode}; (b) the perimeter term $\widehat P^{n}$, with $T_c$ and $t_1$ marked ($t_2$, two steps after $t_1$, coincides with it at this scale); (c) the relative area error in percent.}
    \label{fig:num-annulus}
\end{figure}

For $t\le0.75T_c$, the radii follow \cref{eq:num-annulus-ode} with maximal relative errors $0$ (disk), $\sci{5.9}{-4}$ (outer circle) and $\sci{3.4}{-3}$ (hole).
The gap between the hole and the disk reaches $\ell$ at $t_1$ (step $2047$), where the two components are merged in the metric and, after the auxiliary step of \cref{app:implementation}, in the constraints.
The pair is held fixed and removed two steps later, at $t_2$ (step $2049$), which rescales the outer loop by $c^\ast=0.9902$. The merged area is conserved to $\sci{4.0}{-5}$ over the whole run (without the transfer of the signed area the removal would change it by $2.0\%$).
Afterward the outer circle is exactly stationary (the minimization returns the zero displacement over the remaining $350$ steps) at a radius whose relative difference from $R_\textup{eq}$ is $\sci{1.1}{-4}$, the difference between the radius $\sqrt{A/\pi}$ given by the discrete area and its continuum value.

\subsubsection{Two ellipses with local contact}

The two ellipses have semi-axes $(a,b)=(0.4,1)$ and centres $(\mp0.45,0)$, so that the initial gap is $0.1$. The run uses $N=376$ and takes $4400$ steps.
There is no closed-form solution. The checks are the conservation of the area and of the barycentre of each component, the monotone decrease of $\widehat P^{n}$ and the relaxation of the merged loop towards the disk of the same area, of radius $R_\textup{eq}=\sqrt{2ab}=0.894427$.
\Cref{fig:num-two-ellipses} shows the run with $t_1$ (switch of the visibility), $t_2$ (merging of the constraints) and $t_3$ (reconnection).

\begin{figure}
    \centering
    \includegraphics[width=\linewidth]{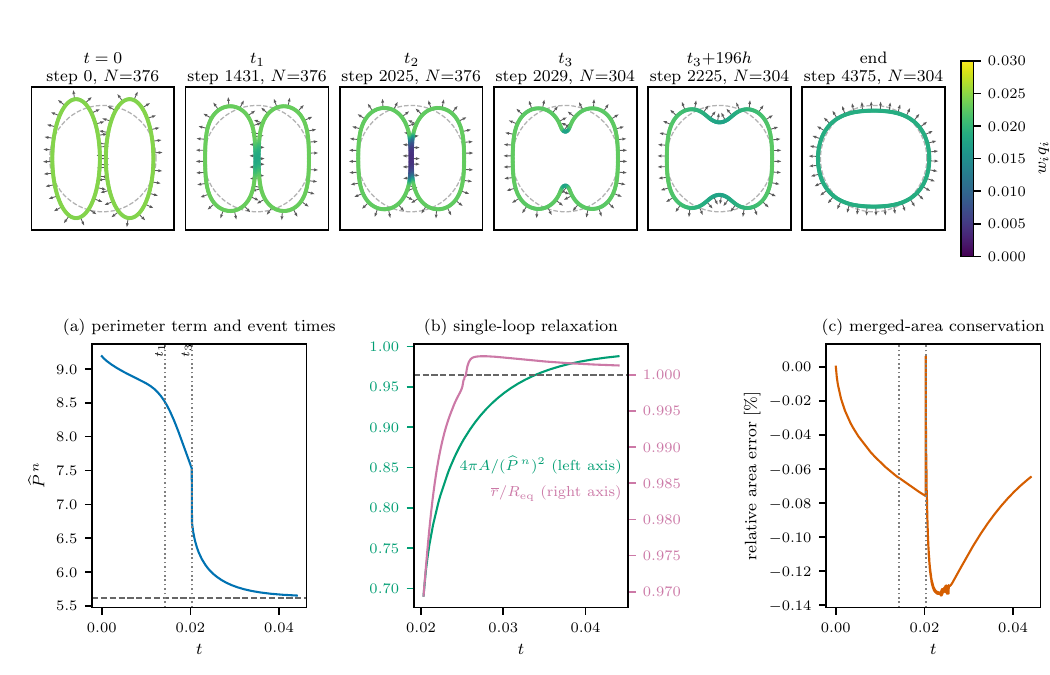}
    \caption{Top: point cloud at $t=0$, at the switch $t_1$, at the merging $t_2$, at the reconnection $t_3$ (the state immediately after it), about $200$ steps later and near the end of the run, coloured by $w_iq_i$ (the dashed circle has radius $R_\textup{eq}$). Bottom: (a) perimeter term $\widehat P^{n}$ with the times $t_1$ and $t_3$ ($t_2$ and $t_3$ coincide at this scale) and the dashed line $2\pi R_\textup{eq}$; (b) circularity and mean radius of the merged loop; (c) relative error of the (merged) area in percent.}
    \label{fig:num-two-ellipses}
\end{figure}

At $t_1$ (step $1431$), the gap reaches $0.55\sigma$ and the visibility on the facing arcs is switched to \cref{eq:application-cc-visibility}. The auxiliary step changes the perimeter term by $\sci{4.4}{-16}$ and the areas and first moments by less than $10^{-8}$.
The gap reaches $\ell$ at $t_2$ (step $2025$), where the two components are merged in the metric, the particles of the facing arcs ($46$) are held fixed, and the constraints are merged after the auxiliary step (the steps with and without the merging differ by less than $10^{-4}\ell$ in position and $\sci{6.4}{-5}$ in angle).
The reconnection at $t_3$ (step $2029$) removes $78$ particles and inserts six.
During the subsequent relaxation of the merged loop, the circularity increases from $0.69$ to $0.99$, $\overline r$ reaches $R_\textup{eq}$ within $0.13\%$, and $\widehat P^{n}$ is within $0.6\%$ of $2\pi R_\textup{eq}$ at the end, where the maximal normal displacement per step is $\sci{2.5}{-5}$.
Before the merging, the barycentre of each ellipse is conserved to $\sci{4.1}{-6}$ by the centred rows of \cref{subsec:application-normal-graph}.
The relative error $e_A^n\coloneqq(A^n-A^0)/A^0$ of the merged area satisfies $\max_n\abs{e_A^n}=\sci{1.3}{-3}$, attained about $340$ steps after the reconnection, with $e_A^n=-\sci{7.6}{-4}$ just before the reconnection and $e_A^n=-\sci{6.5}{-4}$ at the final step.

At the final step, $\abs{R_+-R_\textup{eq}}/R_\textup{eq}=\sci{1.1}{-4}$ in the annulus run, where $R_\textup{eq}=R_+(T_c)$, and $(\widehat P^{n}-2\pi R_\textup{eq})/(2\pi R_\textup{eq})=0.6\%$ in the two-ellipse run, where $R_\textup{eq}=\sqrt{2ab}$. These are discretization errors whose dependence on $\ell$, $\sigma$ and $\varepsilon_{\textup{reg}}$ we have not measured, and the dense factorizations make the cost of a step grow as $N^3$.
Simultaneous multiple contacts are outside the scope of the events implemented here.

\appendix
\section{Implementation details}
\label{app:implementation}
\suppressfloats[t]%

This appendix describes how the scheme of \cref{sec:application} is carried out in the computations of \cref{sec:numerical}, namely the order of the operations of one step, the choice of the parameters, the treatment of several loops and the events.
It states the rules, and the values of all thresholds and tolerances are listed in the documentation of the code (see the footnote in \cref{sec:numerical}).

\paragraph{One step.}
\Cref{alg:application-mm-step} lists the operations of one minimizing movement step, where $\mathsf{Red}$ denotes the redistribution of \cref{subsec:application-tangential-redistribution}. The events of \cref{subsec:application-events} are applied to the state before the step, in the order (i)--(iv).

\begin{algorithm}[t]
\caption{One minimizing movement step $\mathsf{MMStep}$}
\label{alg:application-mm-step}
\begin{algorithmic}[1]
\Require $\pab{x_i^{n-1},\alpha_i^{n-1}}_{i=1}^N$; $h,\sigma,\delta,\tau,\delta_{\textup{red}},h_{\textup{red}}>0$.
\Ensure $V_N^n$.
\State $(w_i^{n-1})_i\gets N^{-1}\times$\cref{eq:weights-from-BB};\quad $(q_i^{n-1})_i\gets$ \cref{eq:application-step-visibility} or \cref{eq:application-cc-visibility}
\State $A^{n-1},B^{n-1}\gets$ \cref{eq:application-angle-matrices};\quad $C^{n-1}\gets$ \cref{eq:application-conservation-rows};\quad $Q^{n-1}\gets$ orthonormal basis of $\ker C^{n-1}$
\State $DS\Lambda\gets$ \cref{eq:application-bie-bordered,eq:application-bie-wasserstein} (per component) \Comment{boundary integral metric}
\State $y^n\gets\argmin$ \cref{eq:application-discrete-mm-step};\quad $s^n\gets Q^{n-1}y^n$
\State $(x_i^n,\alpha_i^n)_i\gets$ \cref{eq:application-position-update,eq:application-angle-constraint}
\State $(x_i^n,\alpha_i^n)_i\gets\mathsf{Red}\pab{(x_i^n,\alpha_i^n,q_i^{n-1})_i}$ \Comment{tangential redistribution}
\State $(w_i^n)_i\gets N^{-1}\times$\cref{eq:weights-from-BB};\quad \Return $V_N^n=\sum_iw_i^n\delta_{(x_i^n,u_i^n)}$
\end{algorithmic}
\end{algorithm}

\paragraph{Parameters.}
The density bandwidth and the truncation level of \cref{subsec:application-discrete-state} are chosen once, from the initial point cloud, by the rule
\begin{equation}
\label{eq:application-delta-tau-rule}
\begin{aligned}
\delta \coloneqq \median_i r_i^{(10)},
&&
\tau \coloneqq \frac{2\mlfr{}(0)}{NZ_{\mlfr{}}^{(1)}\delta},
\end{aligned}
\end{equation}
where $r_i^{(10)}$ is the distance from $x_i$ to its 10th nearest neighbour, $N$ is the initial number of particles and $Z_{\mlfr{}}^{(1)}$ is the normalization constant of the density estimator $\theta_{\delta,N}$ of \cref{sec:estimation} for $d=2$. In the runs of \cref{sec:numerical}, the scales are ordered as $\ell<\sigma<\delta$.

With the notation of \cref{subsec:application-linearized-wasserstein}, the density of the boundary integral metric is computed from
\begin{equation}
\label{eq:application-bie-bordered}
\begin{pmatrix}\widetilde K & a_{\textup{L}}\\ a_{\textup{R}}\transpose & 0\end{pmatrix}
\begin{pmatrix}D^{\nicefrac12}\lambda\\ \mu\end{pmatrix}
=
\begin{pmatrix}D^{\nicefrac12}V(s)\\ 0\end{pmatrix},
\end{equation}
where $a_{\textup{L}}$ and $a_{\textup{R}}$ are the left and right singular vectors of the smallest singular value of $\widetilde K$.
The regularization scale $\varepsilon_{\textup{reg}}$ is the smallest one at which the quadratic form \cref{eq:application-bie-wasserstein} is positive definite on $\ker C^{n-1}$ in the computations of \cref{sec:numerical}. On a circle sampled at the spacing $\ell$, it lowers the form of $v=\cos k\varphi$ by $0.6\%$, $1.2\%$ and $3\%$ for $k=1,3,8$.
The run is stopped when the second smallest singular value of $\widetilde K$, the smallest eigenvalue of the form on $\ker C^{n-1}$ or the residual $\abs{\mu}$ of \cref{eq:application-bie-bordered} relative to $\abs{D^{\nicefrac12}V(s)}$ leaves its bound.
We solve the step \cref{eq:application-discrete-mm-step} by a trust-region Newton conjugate-gradient method with automatic differentiation.

\paragraph{Loops and components.}
The point cloud is organized in loops (\cref{subsec:application-events}), the closed curves of the initial data, and the events reassign whole loops.
The weights \cref{eq:weights-from-BB}, the angle matrices \cref{eq:application-angle-matrices} and the density $\theta^{\textup{red}}$ are computed from the particles of each loop only, and the redistribution step of a loop $\lambda$ is multiplied by the factor $r_\lambda$ of \cref{eq:application-step-visibility}.
The loops are assigned to the connected components of the phase by the regularized winding numbers
\[
\omega_\lambda(p)
\coloneqq
\frac{1}{2\pi}\sum_{i\in\lambda}w_i^{n-1}\frac{u_i^{n-1}\cdot(x_i^{n-1}-p)}{\abs{x_i^{n-1}-p}^2+\ell^2/4}
\]
of the loops $\lambda$ about points $p$ placed on the inner side of each loop at depths of a few $\ell$.
Two loops bound the same component exactly when the rounded winding numbers of all loops about their points coincide, and the run is stopped when the rounding is ambiguous.
The conservation rows \cref{eq:application-conservation-rows} and the system \cref{eq:application-bie-bordered} are set up for each component $b$ from the particles of its loops, and we write the rows $\delta A_b$ and $\delta M_b$.

\paragraph{Visibility of the step.}
When the cloud has several loops, the visibility of the step from $n-1$ to $n$, computed from $V_N^{n-1}$ as in \cref{eq:discrete-visibility-density}, is renormalized loop by loop,
\begin{equation}
\label{eq:application-step-visibility}
\begin{aligned}
q_i^{n-1}
\coloneqq
q_i^{\textup{loop}}
\coloneqq
r_{\lambda}q_i^{\textup{self}},
&&
r_\lambda
\coloneqq
\frac{\sum_{j\in\lambda}w_j q_j}{\sum_{j\in\lambda}w_j q_j^{\textup{self}}},
\end{aligned}
\end{equation}
where $q_i^{\textup{self}}$ is the visibility computed from the particles of the loop $\lambda\ni i$ only, and $q_i$ is the full visibility, that is, the empirical visibility \cref{eq:discrete-visibility-density} of the whole cloud.
The factor $r_\lambda$ carries the visible mass of the loop, $\sum_{j\in\lambda}w_jq_j^{\textup{loop}}=\sum_{j\in\lambda}w_jq_j$, with the profile of the isolated loop. For loops at distance larger than $\sigma$, one has $q^{\textup{loop}}=q$.

\paragraph{Events.}
The dynamics, the conservation rows and the visibility use no ordering of the points, and two approaching loops $\lambda$, $\lambda'$ simply move until their facing arcs coincide.
The coincident arcs then form a discrete counterpart of the hidden boundary of \cref{def:hidden-boundary}. Once the visibility of the step vanishes on them, the minimization \cref{eq:application-discrete-mm-step} becomes ill-conditioned along their displacements, while the conservation rows still treat the two loops as separate components.
The events of \cref{subsec:application-events} lead to a single loop with a single set of conservation rows, in the following order.
Events (i)--(iii) are applied only after an \emph{auxiliary step}, in which one minimizing movement step is computed from the same state with and without the change, and the differences in the objective, in the areas and in the first moments must stay below tolerances.

\paragraph{(i) Switch of the visibility.}
Within distance $\sigma$ of another loop, the full visibility $q$ dips on the facing arcs, while $q^{\textup{loop}}$ applies one factor to the whole loop.
Let $I_\lambda$ be the set of particles of $\lambda$ whose $\sigma$-neighbourhood meets $\lambda'$.
When the gap between the two loops falls below $0.55\sigma$ and $I_\lambda$, $I_{\lambda'}$ are single arcs, the renormalization loop by loop of \cref{eq:application-step-visibility} is replaced on these arcs by the renormalization arc by arc,
\begin{equation}
\label{eq:application-cc-visibility}
\begin{aligned}
q_i^{\textup{arc}}
\coloneqq
r_{I_\lambda}q_i^{\textup{self}}
\quad (i\in I_\lambda),
&&
r_{I_\lambda}
\coloneqq
\frac{\sum_{j\in I_\lambda}w_jq_j}{\sum_{j\in I_\lambda}w_jq_j^{\textup{self}}},
\end{aligned}
\end{equation}
and $q_i^{\textup{arc}}\coloneqq q_i^{\textup{self}}$ elsewhere (there $q=q^{\textup{self}}$, since the other loop is farther than $\sigma$). From then on, the visibility of the step is $q^{\textup{arc}}$. It conserves the visible mass of each arc, coincides with $q^{\textup{loop}}$ when the arc is the whole loop and introduces no new length scale.

\paragraph{(ii) Merging of the constraints.}
When the gap falls below $\ell$, the metric treats the two components as one. The particles of each loop that have a partner on the other loop within distance $1.5\ell$ with an almost opposite normal are excluded from the collocation and held fixed by unit rows added to $C^{n-1}$.
The constraints are merged when, in addition, these particles form one arc on each loop, the paired normals are anti-parallel, and the masses and the currents of the two arcs balance, each up to a threshold. In the concentric benchmark of \cref{sec:numerical}, where the contact is simultaneous along a circle, the same quantities are evaluated on the whole loops.
The rows $\delta A_{b_1},\delta A_{b_2},\delta M_{b_1},\delta M_{b_2}$ of the two components $b_1,b_2$ are then replaced by $\delta A_{b_1}+\delta A_{b_2}$ and the centred $\delta M_{b_1}+\delta M_{b_2}$.

\paragraph{(iii) Removal of the cancelling pair.}
After the merging of the constraints, the coincident arcs form the cancelling pair of \cref{subsec:application-events}, and the two remaining events remove it once the gap has stopped decreasing for three consecutive steps.
If the pair consists of two entire loops (the disk and the hole of the annulus in \cref{sec:numerical}), it is removed, and its signed area $A_\textup{pair}=A_\textup{disk}-A_\textup{hole}$ is transferred to the surviving outer loop by the homothety $x_i\mapsto z+c^\ast(x_i-z)$ about its barycentre $z$, with $c^\ast=\sqrt{(A_\textup{out}+A_\textup{pair})/A_\textup{out}}$, where the areas are those of the polygons through the particles. The factor $c^\ast$ is chosen so that the merged area is conserved.

\paragraph{(iv) Reconnection of the arcs.}
If the pair consists of two facing arcs (the local contact of two convex loops), the particles of each arc within distance $0.08\simeq3.3\ell$ of the other loop are removed, and the two remaining open arcs are reconnected crosswise into a single loop by two cubic Hermite arcs sampled at the existing spacing. A normal displacement of the new loop of least norm then restores the merged area and first moment. Of the two possible reconnections, the one that gives a single simple loop is kept.
The ordering of the points along the loops, needed for the criterion of (ii), the removal and the reconnection only, is derived from the state at the event and is not used by the subsequent steps. The single loop then evolves with $q^{\textup{loop}}=q$.

\paragraph{Thresholds.}
All thresholds were fixed in advance and were not tuned to the reported runs. They are specific to the mollifier $\mlfr{\sigma}$ and the initial data used here, and their transfer to other data has not been tested.

\section{Proof of \cref{prop:imported-scalar-carrier-theorem}}
\label{sec:proofs}

The proof combines the bias and fluctuation estimates of the density estimator (\cref{lem:density-bias,lem:density-at-sample}, which correspond to \citealp[Propositions 6.12 and 3.5]{Boricaud_Buet25}) with a comparison with the empirical measure weighted by the true inverse density (\cref{lem:weighted-empirical}), which avoids the random Lipschitz test functions of \citet[Section 4]{Boricaud_Buet25}.

Throughout this appendix, we write
\[
\begin{aligned}
\theta_\delta(x)\coloneqq\frac{1}{Z_{\mlfr{}}^{(d-1)}\delta^{d-1}}\int_\Gamma\mlfr{}\pab{\frac{x-y}{\delta}}\d\mu_X(y),
&&
x\in\R^d,
\end{aligned}
\]
for the mollified density, so that $\Expectation{\theta_{\delta,N}(x)}=\theta_\delta(x)$ for every deterministic $x$.
We also fix $\delta_0\in(0,1]$ such that $\widetilde R\delta_0\le r_\Gamma$, $R_0\delta_0/c_{\textup{s}}\le r_{\textup{sep}}$ and $2R_0\delta_0\le r_2$, where $\widetilde R\coloneqq\max\{\overline R,2R_0\}$ with $\overline R=R_0/c_{\textup{s}}$ as in \cref{lem:continuous-quantitative-consistency}, and $r_2$ is the radius of the graph representations introduced in the proof of \cref{lem:density-bias}.

\begin{lemma}[Bias of the mollified density]
\label{lem:density-bias}
Under \cref{ass:kernel,ass:piecewise-regular-support,ass:sampling-density}, there exists $C>0$ such that
\(
\int_\Gamma\abs{\theta_\delta-\theta}\d\mu_X\le C\delta^{\alpha\wedge\gamma}
\)
for every $\delta\in(0,\delta_0]$.
\end{lemma}

\begin{proof}
Since each $\Gamma_l$ is a compact $C^{1,\alpha}$ manifold with boundary, there exist $r_2\in(0,r_\Gamma]$ and $C_g>0$ such that, for every $l$ and every $x\in\Gamma_l$, there is $g_x\in C^{1,\alpha}(B^{d-1}_{r_2})$ on the ball $B^{d-1}_{r_2}$ of radius $r_2$ centred at the origin of the tangent plane $T_x\Gamma_l$ with $g_x(0)=0$, $\nabla g_x(0)=0$ and $\abs{\nabla g_x(u)}\le C_g\abs{u}^\alpha\le1$, whose graph $M_x\coloneqq\Set{x+(u,g_x(u)) | u\in B^{d-1}_{r_2}}$ (in the coordinates $T_x\Gamma_l\times(T_x\Gamma_l)^\perp$) contains $\Gamma_l\cap B_{r_2}(x)$. This is the local graph representation of a manifold with boundary, extended beyond $\partial\Gamma_l$, with constants uniform in $x$ by compactness.
In particular, $\abs{g_x(u)}\le\min\{C_g\abs{u}^{1+\alpha},\abs{u}\}$ on $B^{d-1}_{r_2}$.
Let $D_x\subset B^{d-1}_{r_2}$ be the projection of $\Gamma_l\cap B_{r_2}(x)$ onto $T_x\Gamma_l$, so that $\Gamma_l\cap B_{r_2}(x)=\Set{x+(u,g_x(u)) | u\in D_x}$.
Then $D_x\supset B^{d-1}_{\rho_x}$ with $\rho_x\coloneqq\frac12\min\{r_2,\dist(x,\partial\Gamma_l)\}$ ($\rho_x\coloneqq r_2/2$ if $\partial\Gamma_l=\emptyset$).
If $\rho_x=0$, there is nothing to prove, so let $\rho_x>0$.
Indeed, a point $u\in D_x\cap B^{d-1}_{\rho_x}$ corresponds to $y\in\Gamma_l$ with $\abs{y-x}\le\sqrt2\abs{u}<2\rho_x\le\dist(x,\partial\Gamma_l)$, so that $y$ is an interior point of $\Gamma_l$ and, by invariance of domain, a neighbourhood of $y$ in $M_x$ lies in $\Gamma_l$. Hence $D_x\cap B^{d-1}_{\rho_x}$ is open in $B^{d-1}_{\rho_x}$.
It is also closed in $B^{d-1}_{\rho_x}$, being the trace of the projection of the compact set $\Gamma_l\cap\overline{B_{\sqrt2\rho_x}(x)}$, and it contains $0$, so that it is all of $B^{d-1}_{\rho_x}$.
Let $\Gamma_\delta^{\textup{sing}}\coloneqq\Set{x\in\Gamma | \dist(x,\Sigma)\le\widetilde R\delta}$ and $\Gamma_\delta^{\textup{reg}}\coloneqq\Gamma\setminus\Gamma_\delta^{\textup{sing}}$ (the sets of \cref{lem:continuous-quantitative-consistency} with $\widetilde R\delta$ in place of $\overline R\sigma$).
By \cref{eq:kernel-comparison,eq:ahlfors-regularity} (applicable since $R_0\delta\le r_\Gamma$), $0\le\theta_\delta\le C$ on $\Gamma$, so that $\abs{\theta_\delta-\theta}\le C$ on $\Gamma$, and $\mu_X(\Gamma_\delta^{\textup{sing}})\le C\delta$ by \cref{eq:singular-neighbourhood-measure}, since $\widetilde R\delta\le r_\Gamma$.

Let $x\in\Gamma_\delta^{\textup{reg}}\cap\Gamma_l$.
As in the proof of \cref{lem:continuous-quantitative-consistency}, \cref{eq:branch-separation-global} and $\delta\le\delta_0$ give $\dist(x,\Gamma_{l'})\ge R_0\delta$ for $l'\ne l$, so that $B_{R_0\delta}(x)\cap\Gamma=B_{R_0\delta}(x)\cap\Gamma_l$. Moreover, $\dist(x,\partial\Gamma_l)\ge\dist(x,\Sigma)>\widetilde R\delta\ge2R_0\delta$ with $2R_0\delta\le r_2$ gives $\rho_x\ge R_0\delta$, so that $B_{R_0\delta}(x)\cap\Gamma_l$ lies in the graph of $g_x$ over $B^{d-1}_{R_0\delta}\subset D_x$, where $\theta=\theta_l\in C^{0,\gamma}(\Gamma_l)$ up to the $\measure$-null set $\Sigma$.
We write $\theta_\delta(x)-\theta(x)=\nicefrac{(I_1+\theta(x)I_2)}{(Z_{\mlfr{}}^{(d-1)}\delta^{d-1})}$ with
\[
\begin{aligned}
I_1\coloneqq\int_{\Gamma_l}\mlfr{}\pab{\frac{x-y}{\delta}}\pab{\theta(y)-\theta(x)}\dH{d-1}(y),&&
I_2\coloneqq\int_{\Gamma_l}\mlfr{}\pab{\frac{x-y}{\delta}}\dH{d-1}(y)-Z_{\mlfr{}}^{(d-1)}\delta^{d-1}.
\end{aligned}
\]
The H\"older continuity of $\theta_l$ and \cref{eq:ahlfors-regularity} give $\abs{I_1}\le C\delta^{d-1+\gamma}$.
For $I_2$, we parametrize $y=x+(u,g_x(u))$, $u\in B^{d-1}_{R_0\delta}$. Since $\mlfr{}((x-y)/\delta)$ vanishes for $\abs{u}>R_0\delta$, we obtain
\[
\int_{\Gamma_l}\mlfr{}\pab{\frac{x-y}{\delta}}\dH{d-1}(y)=\int_{B^{d-1}_{R_0\delta}}\mlfr{}\pab{\frac{(u,g_x(u))}{\delta}}J_x(u)\d u
\]
with the area element $J_x(u)\coloneqq\sqrt{\det\pab{I+\nabla g_x(u)\transpose\nabla g_x(u)}}$, while $\int_{B^{d-1}_{R_0\delta}}\mlfr{}((u,0)/\delta)\d u=Z_{\mlfr{}}^{(d-1)}\delta^{d-1}$ since $\mlfr{}$ is radial.
On $\abs{u}\le R_0\delta$, we have $\abs{\mlfr{}((u,g_x(u))/\delta)-\mlfr{}((u,0)/\delta)}\le\Lip(\mlfr{})C_g\abs{u}^{1+\alpha}/\delta\le C\delta^\alpha$ and $0\le J_x(u)-1\le C\abs{\nabla g_x(u)}^2\le C\delta^{2\alpha}$, so that $\abs{I_2}\le C\delta^{d-1+\alpha}$ after integration over $B^{d-1}_{R_0\delta}$.
Therefore $\abs{\theta_\delta-\theta}\le C\delta^{\alpha\wedge\gamma}$ on $\Gamma_\delta^{\textup{reg}}$, and $\int_\Gamma\abs{\theta_\delta-\theta}\d\mu_X\le C\delta^{\alpha\wedge\gamma}+C\mu_X(\Gamma_\delta^{\textup{sing}})\le C\delta^{\alpha\wedge\gamma}$ since $\alpha\wedge\gamma,\delta\le1$.
\end{proof}

\begin{lemma}[Density estimate at a sample point]
\label{lem:density-at-sample}
Under \cref{ass:kernel,ass:piecewise-regular-support,ass:sampling-density,ass:marked-iid-model}, there exists $C>0$ such that, for every $\delta\in(0,\delta_0]$ and every $N\ge2$,
\[
\Expectation{\abs{\theta_{\delta,N}(X_1)-\theta(X_1)}}
\le
C\pab{\delta^{\alpha\wedge\gamma}+\frac{1}{\sqrt{N\delta^{d-1}}}+\frac{1}{N\delta^{d-1}}}.
\]
\end{lemma}

\begin{proof}
Conditionally on $X_1=x$, we have $\theta_{\delta,N}(x)=\mlfr{}(0)/(NZ_{\mlfr{}}^{(d-1)}\delta^{d-1})+\frac{N-1}{N}\widetilde\theta(x)$, where $\widetilde\theta(x)\coloneqq((N-1)Z_{\mlfr{}}^{(d-1)}\delta^{d-1})^{-1}\sum_{j=2}^N\mlfr{}((x-X_j)/\delta)$ is the average of $N-1$ independent random variables, each with mean $\theta_\delta(x)$ and, by \cref{eq:ahlfors-regularity}, with variance at most $C\delta^{1-d}$ (as in the proof of \citealp[Proposition 3.5]{Boricaud_Buet25}).
Hence $\widetilde\theta(x)$ has mean $\theta_\delta(x)$ and variance at most $C/((N-1)\delta^{d-1})$, and the Cauchy--Schwarz inequality with $N-1\ge N/2$ gives
\[
\ExpectationGiven{\abs{\theta_{\delta,N}(X_1)-\theta(X_1)}}{X_1=x}
\le
\frac{C}{N\delta^{d-1}}
+\frac{\theta_\delta(x)}{N}
+\frac{C}{\sqrt{N\delta^{d-1}}}
+\abs{\theta_\delta(x)-\theta(x)}.
\]
Integrating with respect to $\mu_X(\d x)$, using $\theta_\delta\le C$ on $\Gamma$ and \cref{lem:density-bias} gives the claim.
\end{proof}

The next lemma bounds the bounded-Lipschitz distance between a weighted empirical measure and its mean.
The proof partitions $\Gamma$ into cubes of side $\rho$ and bounds the fluctuation of the masses of the cubes by the Cauchy--Schwarz inequality, as in \citet[Proposition 3.1]{Dudley1969}.
A single scale $\rho$ suffices for \cref{prop:imported-scalar-carrier-theorem}, whereas \citet[Theorem 3.2]{Dudley1969} chains such partitions over all scales.
\begin{lemma}[Weighted empirical measures]
\label{lem:weighted-empirical}
Under \cref{ass:piecewise-regular-support,ass:sampling-density,ass:marked-iid-model}, let $G>0$ and let $g\colon\Gamma\to[0,G]$ be Borel measurable.
Then there exists $C>0$, depending only on $\Gamma$ and $d$, such that
\[
\begin{aligned}
\Expectation{d_{\textup{BL},\overline\Omega}\pab{\frac1N\sum_{i=1}^Ng(X_i)\delta_{X_i},\ g\mu_X}}
\le
CGN^{-\nicefrac{1}{(d+1)}}
&&
\text{for every } N\in\Z_{>0}.
\end{aligned}
\]
\end{lemma}

\begin{proof}
Write $\lambda_N^g\coloneqq\frac1N\sum_ig(X_i)\delta_{X_i}$ and $\rho\coloneqq N^{-\nicefrac{1}{(d+1)}}$.
If $\rho>r_\Gamma$, then $d_{\textup{BL},\overline\Omega}(\lambda_N^g,g\mu_X)\le\lambda_N^g(\Gamma)+G\le2G\le2Gr_\Gamma^{-1}N^{-\nicefrac{1}{(d+1)}}$, so we assume $\rho\le r_\Gamma$.
Among the half-open cubes $\rho(k+[0,1)^d)$, $k\in\Z^d$, which partition $\R^d$, let $\mathcal Q$ be the family of those that meet $\Gamma$, and for $Q\in\mathcal Q$ pick $x_Q\in Q\cap\Gamma$.
The balls $B_\rho(x_Q)$ are contained in the concentric cubes of side $3\rho$, each point of $\R^d$ lies in at most $3^d$ of these enlarged cubes, and $\measure(B_\rho(x_Q))\ge c_\Gamma\rho^{d-1}$ by \cref{eq:ahlfors-regularity}. Hence the number $\#\mathcal Q$ of cubes in $\mathcal Q$ satisfies $\#\mathcal Q\le3^d\measure(\Gamma)/(c_\Gamma\rho^{d-1})$.
Let $\varphi\in\BL(\overline\Omega)$ with $\norm{\varphi}_{\textup{BL}(\overline\Omega)}\le1$ and set $\varphi_\rho\coloneqq\sum_{Q\in\mathcal Q}\varphi(x_Q)\characteristics{Q}$, so that $\abs{\varphi-\varphi_\rho}\le\sqrt d\rho$ on $\Gamma$.
Then
\[
\abs{\int\varphi\d(\lambda_N^g-g\mu_X)}
\le
\sqrt d\rho\pab{\lambda_N^g(\Gamma)+G}
+\sum_{Q\in\mathcal Q}\abs{\lambda_N^g(Q)-(g\mu_X)(Q)},
\]
and the right-hand side does not depend on $\varphi$.
Since $\Expectation{\lambda_N^g(\Gamma)}\le G$ and, for each $Q$, $\lambda_N^g(Q)$ is the average of $N$ independent copies of $g(X_1)\characteristics{Q}(X_1)$ with mean $(g\mu_X)(Q)$, the Cauchy--Schwarz inequality gives
\[
\sum_{Q\in\mathcal Q}\Expectation{\abs{\lambda_N^g(Q)-(g\mu_X)(Q)}}
\le
\sum_{Q\in\mathcal Q}\sqrt{\frac{G^2\mu_X(Q)}{N}}
\le
G\sqrt{\frac{\#\mathcal Q}{N}}
\le
G\sqrt{\frac{3^d\measure(\Gamma)}{c_\Gamma}}\rho^{-(d-1)/2}N^{-1/2},
\]
where we used $\sum_Q\mu_X(Q)=1$.
Taking the supremum over $\varphi$ and then the expectation, and using $\rho^{-(d-1)/2}N^{-1/2}=\rho$, we obtain the claim.
\end{proof}

\begin{proof}[Proof of \cref{prop:imported-scalar-carrier-theorem}]
Set $\lambda_N^\circ\coloneqq\frac1N\sum_{i=1}^N\theta(X_i)^{-1}\delta_{X_i}$, the empirical measure weighted by the true inverse density, and note that $\theta^{-1}\mu_X=\measure$.
Since $\tau\le\theta_-\le\theta$, we have $\Phi_\tau(\theta(X_i))=1/\theta(X_i)$, and $\Phi_\tau$ is Lipschitz on $[0,\infty)$ with constant $4/\tau^2$.
For every $\varphi$ with $\norm{\varphi}_{\textup{BL}(\overline\Omega)}\le1$,
\[
\abs{\int\varphi\d(\measureMassEmpirical{\delta,N}-\lambda_N^\circ)}
\le
\frac1N\sum_{i=1}^N\abs{\Phi_\tau(\theta_{\delta,N}(X_i))-\Phi_\tau(\theta(X_i))}
\le
\frac{4}{\tau^2}\frac1N\sum_{i=1}^N\abs{\theta_{\delta,N}(X_i)-\theta(X_i)},
\]
and the samples are exchangeable, so that \cref{lem:density-at-sample} gives, for $\delta\le\delta_0$ and $N\ge2$,
\[
\Expectation{d_{\textup{BL},\overline\Omega}(\measureMassEmpirical{\delta,N},\lambda_N^\circ)}
\le
\frac{4}{\tau^2}\Expectation{\abs{\theta_{\delta,N}(X_1)-\theta(X_1)}}
\le
C\pab{\delta^{\alpha\wedge\gamma}+\frac{1}{\sqrt{N\delta^{d-1}}}+\frac{1}{N\delta^{d-1}}}.
\]
\Cref{lem:weighted-empirical} with $g=\theta^{-1}\le\theta_-^{-1}$ gives $\Expectation{d_{\textup{BL},\overline\Omega}(\lambda_N^\circ,\measure)}\le CN^{-\nicefrac{1}{(d+1)}}$ and the triangle inequality yields
\[
\varepsilon(\delta,N)
\le
C\pab{\delta^{\alpha\wedge\gamma}+\frac{1}{\sqrt{N\delta^{d-1}}}+\frac{1}{N\delta^{d-1}}+N^{-\nicefrac{1}{(d+1)}}}.
\]
For $\delta=\delta_N$ and $r$ as in \cref{eq:BB-bandwidth-choice}, we have $\delta_N^{\alpha\wedge\gamma}=N^{-r}$ and $N\delta_N^{d-1}=N^{2r}$, so that the first three terms are bounded by $N^{-r}+N^{-r}+N^{-2r}$, and $N^{-\nicefrac{1}{(d+1)}}\le N^{-r}$ because $\alpha\wedge\gamma\le1$.
This proves \cref{eq:BB-rate} for all $N\ge2$ with $\delta_N\le\delta_0$. The finitely many remaining $N$ are absorbed into $C_0$, since $\varepsilon(\delta,N)\le\Expectation{\measureMassEmpirical{\delta,N}(\Omega)}+\measure(\Omega)\le 1/\tau+\measure(\Omega)$ by \cref{eq:weights-bounded}.
\end{proof}

\section*{Acknowledgements}
The authors used the AI systems Claude (Anthropic) and ChatGPT (OpenAI) in the preparation of this work: to draft and edit parts of the text, to check the proofs and the consistency between the manuscript, the code and the numerical results and to write parts of the code that produced the results of \cref{sec:numerical}. They reviewed and verified all content and take full responsibility for it.
T.E. is supported by Agence Nationale de la Recherche (ANR-23-PEIA-0004, PEPR IA PDE-AI), and N.I. by Japan Science and Technology Agency ACT-X (JPMJAX25C2) and Japan Society for the Promotion of Science KAKENHI (JP25K24910).

\printbibliography

\end{document}